\documentclass[11pt]{amsart}
\usepackage{geometry}                
\usepackage{graphicx}
\usepackage{amssymb}
\usepackage{epstopdf}
\usepackage{epsfig,euscript}
\usepackage[usenames]{color}
\newcommand{\cyc}{\mathrm{cyc}}
\newcommand{\red}{\mathrm{red}}
\newcommand{\cont}{\mathrm{cont}}
\newcommand{\sg}{\sigma}
\newcommand{\ind}{\mathrm{ind}}
\newcommand{\lan}{\langle}
\newcommand{\ran}{\rangle}

\definecolor{MyPurple}{rgb}{0.8,0,0.8}

\newtheorem{theorem}{Theorem}
\newtheorem{lemma}[theorem]{Lemma}
\newtheorem{corollary}[theorem]{Corollary}

\newcommand{\fig}[2]{\begin{figure}[ht]
\centerline{\scalebox{.66}{\epsfig{file=#1.eps}}}
\caption{#2}
\label{fig:#1}
\end{figure}}

\title{Frame patterns in $n$-cycles}

\author{Miles Jones}
\thanks{Instituto de Matem\'atica, Universidad de Talca, Camino Lircay S/N Talca, Chile \texttt{mjones@inst-mat.utalca.cl} \ Supported by FONDECYT (Fondo Nacional de Desarrollo Cient\'{\i}fico y
Tecnol\'ogico de Chile) postdoctoral grant \#3130631.}

\author{Sergey Kitaev}
\thanks{Department of Computer and Information Sciences, University of Strathclyde, Glasgow G1 1XH, UK. Email: \texttt{sergey.kitaev@cis.strath.ac.uk}}

\author{Jeffrey Remmel}
\thanks{Department of Mathematics, University of California, San Diego, La Jolla, CA 92093-0112. USA. Email: \texttt{jremmel@ucsd.edu}}

\begin{document}
\maketitle

\begin{abstract} In this paper, we study the distribution of the number 
of occurrences of the simplest frame pattern, called the $\mu$ pattern, in 
$n$-cycles. Given an $n$-cycle $C$, we say that a pair 
$\langle i,j \rangle$ matches the $\mu$ pattern if $i < j$ and as we traverse 
around $C$ in a clockwise 
direction starting at $i$ and ending at $j$, we never encounter 
a $k$ with $i < k < j$. We say that $ \lan i,j \ran$ is a nontrivial $\mu$-match 
if $i+1 < j$. Also, an $n$-cycle $C$ is incontractible if there is no 
$i$ such that $i+1$ immediately follows $i$ in $C$. 

We show that the number 
of incontractible $n$-cycles in the symmetric 
group $S_n$ is $D_{n-1}$, where $D_n$ is the 
number of derangements in $S_n$. Further, we prove that 
the number of $n$-cycles in $S_n$ with exactly $k$ $\mu$-matches can 
be expressed as a linear combination of binomial coefficients 
of the form $\binom{n-1}{i}$ where $i \leq 2k+1$. We also show that the 
generating function $NTI_{n,\mu}(q)$ 
of $q$ raised to the number of nontrivial $\mu$-matches 
in $C$ over all incontractible $n$-cycles in $S_n$ is a new $q$-analogue of 
$D_{n-1}$, which is different from the $q$-analogues of the derangement
numbers that have been studied by Garsia and Remmel and by Wachs. 
We show that there is a rather surprising connection 
between the charge statistic on permutations due to Lascoux and 
Sch\"uzenberger and our polynomials in that 
the coefficient of the smallest power of $q$ in $NTI_{2k+1,\mu}(q)$ is 
the number of permutations in $S_{2k+1}$ whose charge path is a Dyck path. Finally, we show 
that $NTI_{n,\mu}(q)|_{q^{\binom{n-1}{2} -k}}$
  and $NT_{n,\mu}(q)|_{q^{\binom{n-1}{2} -k}}$ are the number of 
partitions of $k$ for sufficiently large $n$.
\end{abstract}

\section{introduction}

Mesh patterns were introduced in \cite{BrCl} by Br\"and\'en and Claesson, and they were studied in a series of papers (e.g. see \cite{kitlie} by Kitaev and Liese, and references therein). A particular class of mesh patterns is {\em boxed patterns} introduced in \cite{AKV} by Avgustinovich et al., who later suggested to call this type of patterns {\em frame patterns}. The simplest 
frame pattern which is called the $\mu$ pattern is defined 
as follows. Let $S_n$ denote the set of all permutations of $\{1, \ldots, n\}$. Given 
$\sigma=\sigma_1\sigma_2\dots\sigma_n \in S_n$, we 
say that  a pair $\langle \sigma_i,\sigma_j \rangle$ is an 
occurrence of the $\mu$  pattern in $\sg$ if $i<j$,
$\sigma_i<\sigma_j$, and there is no $i < k < j$ 
such that $\sigma_i< \sg_k < \sigma_j$ (the last condition is indicated by the shaded area in Figure \ref{fig:48261357}). The $\mu$ pattern 
is shown in Figure~\ref{fig:meshpattern} using the notation in \cite{BrCl}. 

\begin{figure}[htbp]
  \begin{center}
   \includegraphics[width=0.15\textwidth]{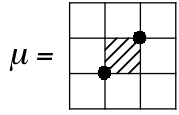}
    \caption{The mesh pattern $\mu$.}
    \label{fig:meshpattern}
  \end{center}
\end{figure}

Similarly, we say that the pair $\langle \sigma_i,\sigma_j \rangle$ is an 
occurrence of the $\mu'$ pattern   in $\sg$ if $i<j$, $\sigma_i>\sigma_j$, 
and there is no $i < k < j$ such that $\sigma_i> \sg_k > \sigma_j$. 
For example, if $\sigma=4~8~2~6~1~3~5~7$, then the occurrences of $\mu$ in 
$\sg$ are
 $$\lan 4,8\ran,\lan 4,6\ran,\lan 4,5\ran,\lan 2,6\ran,\lan 2,3\ran,\lan 6,7\ran,\lan 1,3\ran,\lan 3,5\ran,\lan 5,7\ran$$
and the occurrences of $\mu'$ in $\sg$ are 
$$\lan 4, 2 \ran, \lan 4,3 \ran, \lan 8, 2 \ran, \lan 8,6 \ran, \lan 8,7 \ran, 
\lan 2,1 \ran, \lan 6,1 \ran, \lan 6,3 \ran,  \lan 6, 5 \ran.$$
We let $N_\mu (\sg)$ (resp., $N_{\mu'} (\sg)$) denote the number of occurrences of 
the $\mu$ (resp., $\mu'$) in $\sg$. 
The {\em reverse} of  $\sg = \sg_1 \ldots \sg_n \in S_n$, $\sg^r$, is the permutation $\sg_n \sg_{n-1} \ldots \sg_1$, and the {\em complement} of $\sg$, $\sg^c$,  is the permutation 
$(n+1-\sg_1) (n+1 -\sg_2) \ldots (n+1 - \sg_n)$.  It is 
easy to see that $N_{\mu}(\sg) = N_{\mu'}(\sg^r) = N_{\mu'}(\sg^c)$ and thus, since the reverse and complement are trivial bijections from $S_n$ to itself, studying the distribution of $\mu$-matches in $S_n$  is equivalent to studying 
the distribution of $\mu'$=mathches in $S_n$.

 \begin{figure}[htbp]
  \begin{center}
   \includegraphics[width=0.2\textwidth]{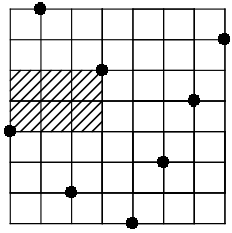}
    \caption{The graph of the permutation 48261357 with the occurrence 
$\lan 4,6 \ran$ highlighted.}
    \label{fig:48261357}
  \end{center}
\end{figure}

If we graph a given permutation $\sg$ as dots on a grid as in Figure \ref{fig:48261357}, then one can see
that each occurrence of $\mu$ is a pair of increasing dots such that there are no dots within the
rectangle created by the original dots. The occurrence $\lan 4,6 \ran$ is highlighted in Figure \ref{fig:48261357}; in particular, there are no dots within the shaded rectangle.

Jones and Remmel studied the distribution of 
cycle-occurrences of classical consecutive patterns in \cite{JR}. See \cite{kit} for a comprehensive introduction to the theory of permutation patterns. In this paper, we shall study the distribution 
of the {\em cycle-occurrences} of $\mu$ in the cycle structure 
of permutations in the symmetric group $S_n$.  That is, 
suppose that we are given 
a $k$-cycle $C$ in a permutation $\sg$ in the symmetric group $S_n$. Then  
we will always write $C =(c_0, \ldots, c_{k-1})$ where $c_0$ is the 
smallest element of the cycle. We will always draw such a 
cycle with $c_0$ at the top and assume that we traverse around the cycle 
in a clockwise direction. For example, if $C = (1,4,6,2,7,5,8,3)$, then we 
would picture $C$ as in Figure \ref{fig:cycle}.  We 
 say that $c_s$ is {\em cyclically between} 
$c_i$ and $c_j$ in $C$, if starting 
at $c_i$, we encounter $c_s$ before we encounter $c_j$ as we traverse 
around the 
cycle in a clockwise direction. Alternatively, we can identify 
$C$ with the permutation  $c_0 c_1 \ldots c_{k-1}$. In this notation,  
$c_s$ is cyclically between
 $c_i$ and $c_j$ if either $i<s<j$, or $j<i<s$, or $s<j<i$. For example, in the cycle 
$C = (1,4,6,2,7,5,8,3)$, 8 is cyclically between 2 and 3. 
Then we say that the pair $\langle c_i,c_j \rangle$ is 
a  {\em cycle-occurrence of $\mu$ in $C$}  if $c_i < c_j$ and there is 
no $c_s$ such that $c_i < c_s < c_j$ and $c_s$ is cyclically 
between $c_i$ and $c_j$.  
Similarly, 
we say that the pair $\langle c_i,c_j \rangle$ is 
a cycle-occurrence of $\mu'$ in $C$, if $c_i > c_j$ and there is 
no $c_s$ such that $c_i > c_s > c_j$ and $c_s$ is cyclically 
between $c_i$ and $c_j$. As is the case with permutations,  
the study of the number of cycle-occurrences of $\mu$ in 
the cycle structures of permutations is equivalent to 
the study of the number of cycle-occurrences of $\mu$ in 
the cycle structures of permutations. That is, given 
the cycle structure $C_1, \ldots, C_k$ of a permutation $\sg \in S_n$, 
the cycle complement of $\sg$, $\sg^{\mbox{cyc-c}}$, is the permutation 
whose cycle structure arises from the cycle structure of 
$\sg$ by replacing each number $i$ by $n+1-i$. For example, 
if $\sg$ consists of the cycle $(1,4,2,6), (3,8) (5,9,7)$, then 
$\sg^{\mbox{cyc-c}}$ consists of the cycles $(9,6,8,4),(7,2),(5,1,3)$. 
It is then easy to see that for all $\sg \in S_n$, $\sg$  has $k$ 
cycle-occurrences of $\mu$ if and only if  $\sg^{\mbox{cyc-c}}$ has $k$ cycle-occurrences of $\mu'$.

Let $\mathcal{C}_n$ be the set of $n$-cycles in $S_n$.  If 
$C=(1,c_1, \ldots,c_{n-1}) \in \mathcal{C}_n$, then  
it is clear that $\langle i,i+1\rangle$ is always a cycle-occurrence of 
$\mu$ in $C$. We shall 
call such cycle-occurrences {\em trivial occurrences of $\mu$} or 
 {\em trivial $\mu$-matches} in $C$ and all other cycle-occurrences of $\mu$ 
in $C$ will be called  {\em nontrivial occurrences of $\mu$} or 
{\em nontrivial $\mu$-matches}  in $C$. 
We let $N_{\mu}(C)$ denote the number of occurrences 
of $\mu$ in $C$ and $NT_{\mu}(C)$ denote the number of 
nontrivial occurrences of $\mu$ in $C$. For 
example, if $C = (1,4,6,2,7,5,8,3)$, the nontrivial occurrences 
of $\mu$  in $C$ are the pairs $\langle 1,4\rangle$, 
$\langle 2,7\rangle$, $\langle 2,5\rangle$, $\langle 4,6\rangle$, 
and $\langle 5,8\rangle$, so that $NT_{\mu}(C) =5$. 
Clearly, if $C = (1,c_1, \ldots, c_{n-1})$ is an $n$-cycle in $\mathcal{C}_n$, 
then $N_{\mu}(C) = (n-1) + NT_\mu(C)$, since for all 
$1 \leq i \leq n-1$, $\lan i, i+1 \ran$ will be a trivial occurrence 
of $\mu$ in $C$.   
If $C$ is a 1-cycle, then $NT_\mu(C) = N_{\mu}(C) =0$. 

 \begin{figure}[htbp]
  \begin{center}
   \includegraphics[width=0.2\textwidth]{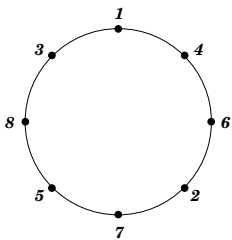}
   \vspace{-1cm}
    \caption{The cycle $C = (1,4,6,2,7,5,8,3)$.}
    \label{fig:cycle}
  \end{center}
\end{figure}

If $\sg_1 \ldots \sg_n$ is any sequence of distinct integers, 
we let $\red (\sg)$ denote the permutation of $S_n$ that 
is obtained by replacing the $i$th largest 
element of $\{\sg_1, \ldots, \sg_n\}$ 
by $i$ for $i =1, \ldots n$. For example, 
$\red (3~7~10~5~2) = 2~4~5~3~1$. Similarly, 
if  $C =(c_0, \ldots, c_{k-1})$ is a $k$-cycle in 
some permutation $\sg \in S_n$, we let 
$\red (C)$ be the $k$-cycle in $S_k$ which is obtained 
by replacing the $i$th largest element of $C$ by $i$. For 
example, if $C = (2, 6,7,3,9)$, then 
$\mbox{red}(C) = (1,3,4,2,5)$. In such a situation, we 
let $NT_\mu(C) = NT_\mu(\mbox{red}(C))$.  Finally, if 
$\sg$ consists of cycles $C^{(1)}, \ldots, C^{(\ell)}$, then we let 
$CNT_\mu(\sg) = \sum_{i =1}^\ell NT_\mu(C^{(i)})$.

We note that if one wishes to study the generating function 
\begin{equation}\label{eq:CNT}
CNT_{\mu}(q,x,t) = 1+ \sum_{n \geq 1} \frac{t^n}{n!} 
\sum_{\sg \in S_n} q^{CNT_{\mu}(\sg)} x^{\mbox{cyc}(\sg)}, 
\end{equation}
where $\mbox{cyc}(\sg)$ denotes the number of cycles in $\sg$, 
then it is enough 
to study the generating function 
\begin{equation}\label{eq:NT}
NT_\mu(q,t) = \sum_{n \geq 1} \frac{t^n}{n!} NT_{n,\mu}(q)
\end{equation}
where $NT_{n,\mu}(q)=\sum_{C \in \mathcal{C}_n}q^{NT_\mu(C)}$.
That is, it easily follows from the exponential formula 
that if $w$ is a weight function on $n$-cycles $C \in \mathcal{C}_n$ 
and for any permutation $\sg$ whose cycle structure is 
$C_1, \ldots, C_k$, we define $w(\sg) = \prod_{i=1}^k w(\red(C_i))$, 
then 
$$1+ \sum_{n \geq 1} \frac{t^n}{n!} \sum_{\sg \in S_n} w(\sg) x^{\cyc(\sg)}= 
e^{x\sum_{n \geq 1} \frac{t^n}{n!} \sum_{C \in \mathcal{C}_n} w(C)}.$$  
Hence 
\begin{equation}
CNT_\mu(q,x,t) = e^{xNT_\mu(q,t)}.
\end{equation}
The main focus of this paper is to study the polynomials 
$NT_{n,\mu}(q)$. 
In Table \ref{tab:incont0} we provide polynomials $NT_{n,\mu}(q)$ for $n\leq 10$ that were calculated using Mathematica.

 \begin{table}[h]
\caption{Polynomials $NT_{n,\mu}(q)$.}\label{tab:incont0}    
    \begin{tabular}{c||*{2}{l}}
$NT_{1,\mu}(q)$ & $1$ \\
 \hline
$NT_{2,\mu}(q)$ & $1$ \\
 \hline
$NT_{3,\mu}(q)$ & $1 + q$ \\
 \hline
$NT_{4,\mu}(q)$ & $1 + 3\,q + q^2 + q^3$ \\
 \hline
$NT_{5,\mu}(q)$ & $1 + 6\,q + 6\,q^2 + 7\,q^3 + 2\,q^4 + q^5 + q^6$ \\
 \hline
$NT_{6,\mu}(q)$ & $1 + 10\,q + 20\,q^2 + 31\,q^3 + 23\,q^4 + 15\,q^5 + 13\,q^6 + 3\,q^7 + 2\,q^8 + q^9 + q^{10}$ \\
 \hline
$NT_{7,\mu}(q)$ & $1 + 15\,q + 50\,q^2 + 106\,q^3 + 135\,q^4 + 126\,q^5 + 119\,q^6 + 66\,q^7 + 46\,q^8 +$\\
&$ 25\,q^9 + 19\,q^{10} + 5\,q^{11} + 3\,q^{12} + 2\,q^{13} + q^{14} + q^{15}$ \\
 \hline
$NT_{8,\mu}(q)$ & $1 + 21\,q + 105\,q^2 + 301\,q^3 + 561\,q^4 + 736\,q^5 + 850\,q^6 + 726\,q^7 + 603\,q^8 +$\\
&$ 418\,q^9 + 299\,q^{10} + 174\,q^{11} + 101\,q^{12} + 65\,q^{13} + 33\,q^{14} + 27\,q^{15} + 7\,q^{16} +$\\
&$ 5\,q^{17} + 3\,q^{18} + 2\,q^{19} + q^{20} + q^{21}$ \\
 \hline
$NT_{9,\mu}(q)$ & $1 + 28\,q + 196\,q^2 + 742\,q^3 + 1870\,q^4 + 3311\,q^5 + 4820\,q^6 + 5541\,q^7 +$\\
&$ 5675\,q^8 + 5007\,q^9 + 4055\,q^{10} + 3093\,q^{11} + 2116\,q^{12} + 1461\,q^{13} + 888\,q^{14} +$\\
&$ 646\,q^{15} + 338\,q^{16} + 217\,q^{17} + 126\,q^{18} + 80\,q^{19} + 44\,q^{20} + 35\,q^{21} +$\\
&$ 11\,q^{22} + 7\,q^{23} + 5\,q^{24} + 3\,q^{25} + 2\,q^{26} + q^{27} + q^{28}$ \\
 \hline
$NT_{10,\mu}(q)$ & $1 + 36\,q + 336\,q^2 + 1638\,q^3 + 5328\,q^4 + 12253\,q^5 + 22392\,q^6 +$\\
&$ 32864\,q^7 + 41488\,q^8 + 45433\,q^9 + 44119\,q^{10} + 40008\,q^{11} + 32781\,q^{12} +$\\
&$ 25689\,q^{13} + 18551\,q^{14} + 13710\,q^{15} + 9137\,q^{16} + 6179\,q^{17} + 3971\,q^{18} + $\\
&$2568\,q^{19} + 1640\,q^{20} + 1098\,q^{21} + 640\,q^{22} + 374\,q^{23} + 251\,q^{24} +$\\
&$ 148\,q^{25} + 100\,q^{26} + 56\,q^{27} + 46\,q^{28} + 15\,q^{29} + 11\,q^{30} + 7\,q^{31} +$\\
&$ 5\,q^{32} + 3\,q^{33} + 2\,q^{34} + q^{35} + q^{36}$ \\
 \hline
\end{tabular}
\end{table}

Given a polynomial $f(x) = \sum_{i=0}^n c_i x_i$, we let $f(x)|_{x^i} = c_i $ 
denote the coefficient of $x^i$ in $f(x)$. It is easy to explain several of 
the coefficients that appear in $NT_{n,\mu}(q)$. For example, for $n\ge2$,
$NT_{n,\mu}(q)|_{q^0}=1$ because $(1,2,\dots,n)$ is 
the only $n$-cycle that has no 
nontrivial occurrences of $\mu$. Similarly, 
$NT_{n,\mu}(q)|_{q^{{n-1\choose2}}}=1$ since  
the only $n$-cycle with $\binom{n-1}{2}$ nontrivial occurrences of $\mu$ 
is $(1,n,n-1,n-2,\dots,2)$.\label{maxcycle} We computed that the 
sequence $(NT_{n,\mu}(q)|_{q^1})_{n \geq 3}$ starts out  
$1,3,6,10,15,21, \ldots $ which leads to a 
conjecture that $NT_{n,\mu}(q)|_{q^1} = \binom{n-1}{2}$ for 
$n \geq 3$.  Similarly, we computed that 
the sequence $(NT_{n,\mu}(q)|_{q^2})_{n \geq 4}$ 
starts out 
$1,6,20,50,105,196, \ldots $ which  is the initial terms of sequence 
A002415 in the {\em On-Line Encyclopedia of Integer Sequences} ({\em OEIS}) 
whose $n$th term is given by the formula $\frac{n^2(n^2-1)}{12}$.  
We will verify 
this by proving that for $n \geq 5$, $NT_{n,\mu}(q)|_{q^2} = 
\binom{n-1}{3} + 2\binom{n-1}{4}$. 

In fact, 
we shall show that for any fixed $k \geq 0$, there are constants $c_0, c_1, 
\ldots, c_{2k+1}$ such that 
$NT_{n,\mu}(q)|_{q^k} = \sum_{i=1}^{2k+1} c_i \binom{n-1}{i}$ 
for all $n \geq 2k+2$. To prove this result, 
we will need to study what we call incontractible $n$-cycles which 
are $n$-cycles $C$ for which there are no integers $i$ such that $i+1$ immediately 
follows $i$ in $C$.  We let $\mathcal{IC}_n$ denote 
the set of incontractible $n$-cycles in $\mathcal{C}_n$. We 
will show that $|\mathcal{IC}_n|$ is $D_{n-1}$ where 
$D_n$ is the number of derangements of $S_n$, i.e. the 
number of $\sg \in S_n$ such that $\sg$ has no fixed points. 
We will  also study the  polynomials 
$$NTI_{n,\mu}(q)=\sum_{C\in \mathcal{IC}_n}q^{NT_\mu(C)}.$$
We have computed Table \ref{tab:poly} for the 
polynomials $NTI_{n,\mu}(q)$ using Mathematica.

    \begin{table}[h]
\caption{Polynomials $NTI_{n,\mu}(q)$.}\label{tab:poly}
\begin{tabular}{c||*{2}{l}}
$NTI_{1,\mu}(q)$ & $1$ \\
 \hline
$NTI_{2,\mu}(q)$ & $0$ \\
 \hline
$NTI_{3,\mu}(q)$ & $q$ \\
 \hline
$NTI_{4,\mu}(q)$ & $q^2 + q^3$ \\
 \hline
$NTI_{5,\mu}(q)$ & $2\,q^2 + 3\,q^3 + 2\,q^4 + q^5 + q^6$ \\
 \hline
$NTI_{6,\mu}(q)$ & $6\,q^3 + 13\,q^4 + 10\,q^5 + 8\,q^6 + 3\,q^7 + 2\,q^8 + q^9 + q^{10}$ \\
 \hline
$NTI_{7,\mu}(q)$ & $5\,q^3 + 27\,q^4 + 51\,q^5 + 56\,q^6 + 48\,q^7 + 34\,q^8 + 19\,q^9 + 13\,q^{10}$\\
&$ + 5\,q^{11} + 3\,q^{12} + 2\,q^{13} + q^{14} + q^{15}$ \\
 \hline
$NTI_{8,\mu}(q)$ & $29\,q^4 + 134\,q^5 + 255\,q^6 + 327\,q^7 + 323\,q^8 + 264\,q^9 + 187\,q^{10} + 139\,q^{11} +$\\
&$ 80\,q^{12} + 51\,q^{13} + 26\,q^{14} + 20\,q^{15} + 7\,q^{16} + 5\,q^{17} + 3\,q^{18} + 2\,q^{19} + q^{20} + q^{21}$ \\
 \hline
$NTI_{9,\mu}(q)$ & $14\,q^4 + 181\,q^5 + 694\,q^6 + 1413\,q^7 + 2027\,q^8 + 2307\,q^9 + 2139\,q^{10} + 1841\,q^{11} +$\\
&$ 1392\,q^{12} + 997\,q^{13} + 652\,q^{14} + 458\,q^{15} + 282\,q^{16} + 177\,q^{17} + 102\,q^{18} + 64\,q^{19} + $\\
&$36\,q^{20} + 27\,q^{21} + 11\,q^{22} + 7\,q^{23} + 5\,q^{24} + 3\,q^{25} + 2\,q^{26} + q^{27} + q^{28}$ \\
 \hline
 $NTI_{10,\mu}(q)$ &$130\,q^5 + 1128\,q^6 + 3965\,q^7 + 8509\,q^8 + 13444\,q^9 + 16918\,q^{10} +$\\
 &$ 18015\,q^{11} + 17121\,q^{12} + 14712\,q^{13} + 11663\,q^{14} + 8784\,q^{15} + 6347\,q^{16} +$\\
 &$ 4406\,q^{17} + 2945\,q^{18} + 1920\,q^{19} + 1280\,q^{20} + 819\,q^{21} + 541\,q^{22} +$\\
 &$ 311\,q^{23} + 206\,q^{24} + 121\,q^{25} + 82\,q^{26} + 47\,q^{27} + 37\,q^{28} + 15\,q^{29} + $\\
 &$11\,q^{30} + 7\,q^{31} + 5\,q^{32} + 3\,q^{33} + 2\,q^{34} + q^{35} + q^{36}$ \\
 \hline
\end{tabular}
\end{table}

We  shall show 
that $NTI_{n,\mu}(q)|_{q^{\binom{n-1}{2} -k}}$
  and $NT_{n,\mu}(q)|_{q^{\binom{n-1}{2} -k}}$ are the number of 
partitions of $k$ for sufficiently large $n$.
We show this by plotting the \emph{non-$\mu$-matches} of  
an $n$-cycle $C \in \mathcal{C}_n$, i.e. the pairs of integers $\langle i,j \rangle$ 
with $i < j$ which are not occurrences of the $\mu$ pattern in $C$, on an 
$n \times n$  grid and shading a cell in the $j$th row and 
$i$th column if and only if 
$\langle i, j \rangle$ is a non-$\mu$-match of $C$. For example, the 
non-$\mu$-matches
of the cycle $C= (1, 8, 7, 6, 5, 11, 4, 3, 10, 2, 9)$ are $\langle1,9\rangle,\langle1,10\rangle,\langle1,11\rangle,\langle2,10\rangle,\langle2,11\rangle,\langle3,11\rangle$
 and $\langle4,11\rangle$. These pairs can be plotted on 
the $11 \times 11$ grid as shown in Figure \ref{fig:fdex1}.
  \begin{figure}[htbp]
  \begin{center}
   \includegraphics[width=0.30\textwidth]{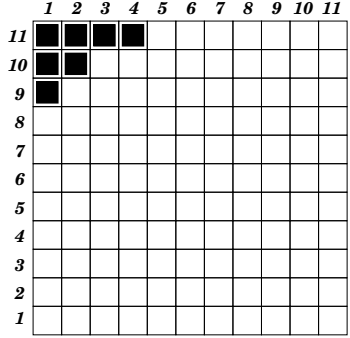}
    \caption{The non-matches of $(1, 8, 7, 6, 5, 11, 4, 3, 10, 2, 9)$ plotted on an $11\times 11$ grid}
    \label{fig:fdex1}
  \end{center}
\end{figure}
The shaded plot in Figure \ref{fig:fdex1} is of the form of the {\em Ferrers diagram} of the integer partition $\lambda=(4,2,1)$. We will show that if $C \in \mathcal{C}_n$ and there are 
fewer than $n-2$ non-matches in $C$, then shaded squares 
in the plot of the non-$\mu$-matches in  $C$ 
will be of the form of a Ferrers diagram of an integer partition. Moreover, 
we shall show that if $C \in \mathcal{C}_n$ 
has fewer than $n-2$ non-matches, then $C$ must be incontractible. 
Hence, it follows that $NTI_{n,\mu}(q)|_{q^{\binom{n-1}{2} -k}}$ and $NT_{n,\mu}(q)|_{q^{\binom{n-1}{2} -k}}$ are the number of 
partitions of $k$ for sufficiently large $n$. 

We will show that $NTI_{n,\mu}(q)|_{q^{k}} = 0$ for 
$n \geq 2k+2$ and that the lowest power of $q$ which  appears in either 
$NTI_{2k,\mu}(q)$ or 
$NTI_{2k+1,\mu}(q)$ is $q^k$. 
We computed 
that the sequence  $(NTI_{2k+1,\mu}(q)|_{q^{k}})_{k \geq 3}$ starts 
out $1,2,5,14, \ldots $ which led us to conjecture that 
$NTI_{2k+1,\mu}(q)|_{q^{k}} =C_k = \frac{1}{n+1} \binom{2n}{n}$ where $C_k$
is the $k$th {\em Catalan number}. We shall prove this conjecture 
and our proof of this conjecture led to a surprising connection  
between the {\em charge statistic on permutations} as defined 
by Lascoux and Sch\"utzenberger \cite{LS} and our problem. 
That is, given a permutation $\sg = \sg_1 \ldots \sg_n$, one 
defines the index, $\mbox{ind}_\sg(\sg_i)$, of $\sg_i$ in $\sg$ as follows. 
First, 
$\mbox{ind}_{\sg}(1) =0$. Then, inductively, 
$\mbox{ind}_{\sg}(i+1) = \mbox{ind}_{\sg}(i)$ 
if $i+1$ appears to the right of $i$ in $\sg$ 
and  $\mbox{ind}_{\sg}(i+1) = 1+\mbox{ind}_{\sg}(i)$
if $i+1$ appears to the left $i$ in $\sg$. 
For example, if $\sg = 1~4~8~5~9~6~2~7~3$ and we use 
subscripts to indicate the index of $i$ in $\sg$, then 
we see that the indices associated with $\sg$ are 
$$\sg =  1_0~4_1~8_2~5_1~9_2~6_1~2_0~7_1~3_0 .$$ 
The charge of $\sg$, $\mbox{ch}(\sg)$, is equal to 
$\sum_{i=1}^n \mbox{ind}_\sg(i)$. Thus, 
for example, if $\sg = 1~4~8~5~9~6~2~7~3$, then $\mbox{ch}(\sg) =8$. 
We will associate a path  which we call the {\em charge path} of 
$\sg$ whose vertices are elements $(i,\mbox{ind}_\sg(\sg_i))$ and 
whose edges are $\{ 
(i,\mbox{ind}_\sg(\sg_i)), (i+1,\mbox{ind}_\sg(\sg_{i+1}))\}$ for $i = 1, 
\ldots, n$. 
For example, if $\sg = 1~4~8~5~9~6~2~7~3$ as above, then 
charge path of $\sg$, which we denote by $\mbox{chpath}(\sg)$,
 is pictured in Figure \ref{fig:chgr}. 
In our particular example, the charge graph of $\sg$ is a {\em Dyck 
path} (a lattice path with steps (1,1) and (1,-1) from the origin (0,0) to $(2n,0)$ that never goes below the $x$-axis) and if $C =(1,4,8,5,9,6,2,7,3)$ is the $n$-cycle induced 
by $\sg$, then $NT_\mu(C) = 4$. This is no accident. That is, 
we shall show that if $C =(1,\sg_2, \ldots, \sg_{2n+1})$ is 
$2n+1$-cycle in $\mathcal{IC}_{2n+1}$, then 
$NT_\mu(C) = n$ if and only if the charge path of 
$\sg = 1 \sg_2 \ldots \sg_{2n+1}$  
is a Dyck path of length $2n$.

\fig{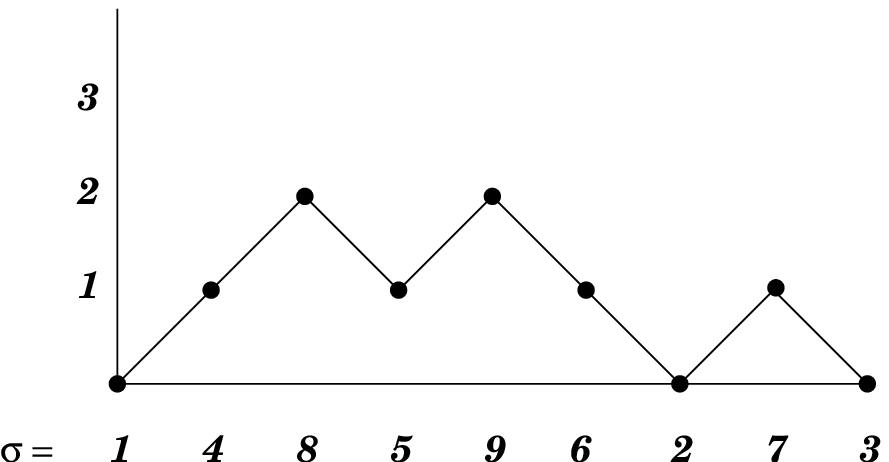}{The charge path of $\sg = 1~4~8~5~9~6~2~7~3$.}

The outline of this paper is as follows. In Section \ref{cont}, we shall 
introduce the notion of the contraction of an $n$-cycle 
and use it to show that for all $n \geq 1$, there 
are positive integers  $c_1, \ldots, c_{2k+1}$ such that 
$NT_{n,\mu}(q)|_{q^k} = \sum_{s= \left\lfloor\frac{3+\sqrt{1+8k}}{2}\right\rfloor}^{2k+1} c_s \binom{n-1}{s-1}$. In Section \ref{incont}, 
we shall study incontractible cycles and the polynomials 
$NTI_{n,\mu}(q)$. We shall show that 
for all $n \geq 1$, $NTI_{2n+1,\mu}(q)|_{q^n} =C_n$ where 
$C_n$ is $n$th Catalan number.  In Section \ref{intpart}, we study 
the non-$\mu$-matches in $n$-cycles of $\mathcal{C}_n$ and use
them to show that $NTI_{n,\mu}(q)|_{q^{\binom{n-1}{2} -k}}$ and $NT_{n,\mu}(q)|_{q^{\binom{n-1}{2} -k}}$ are the number of 
partitions of $k$ for sufficiently large $n$. In Section \ref{concl}, we will state our conclusions and some open problems.

\section{The contraction of a cycle.}\label{cont}

Define a \emph{bond} of an $n$-cycle $C \in \mathcal{C}_n$ 
to be a pair of consecutive integers $(a,a+1)$ such that
$a$ is immediately followed by $a+1$ in the cycle $C$.
For example, the bonds of the cycle $(1,4,5,6,2,7,8,3)$ are the pairs $(4,5)$, $(5,6)$, and $(7,8)$. If $C=(c_1,c_2,\dots,c_n)$ is an $n$-cylce 
where $c_1 =1$, define the sets
 $R_1,R_2,\dots,R_k$ recursively as follows. First, let $1=c_1 \in R_1$. 
Then inductively, if $c_i\in R_j$, then $c_{i+1}\in R_j$ if $c_{i+1}=c_i+1$,
and $c_{i+1}\in R_{j+1}$ otherwise.
For example, for
 the cycle $C=(1,2,4,6,7,8,3,5)$ these sets are $R_1=\{1,2\},R_2=\{4\},R_3=\{6,7,8\},R_4=\{3\},R_5=\{5\}$.
 We will call these sets the \emph{consecutive runs of $C$}.
 Define the \emph{contraction} of a cycle $C$
 with consecutive runs 
  $R_1,R_2,\dots,R_k$ to be 
  $$\cont(C)=\red(\max(R_1),\max(R_2),\dots,\max(R_k)),$$
where $\max(R_i)$ denotes the maximum element in a run $R_i$.  For example, the contraction of $C=(1,2,4,6,7,8,3,5)$ is
  $$\cont(1,2,4,6,7,8,3,5)=\red(2,4,8,3,5)=(1,3,5,2,4).$$
  We say that $C$ \emph{contracts} to $A$ if $\cont(C)=A$.
  Clearly, a cycle $C$ is \emph{incontractible} if $\cont(C)=C$.

We claim that the number of non-trivial $\mu$ matches 
does not change as we pass from  $C$ to $\cont(C)$. That is, 
we have the following theorem. 
\begin{theorem}\label{ntmu}
For any $n$-cycle $C \in \mathcal{C}_n$, 
\begin{equation*}
NT_\mu(C)=NT_\mu(\cont(C)).
\end{equation*}
\end{theorem}
\begin{proof}
We proceed by induction on $n$. The theorem is clearly true for 
$n =1$. Now, suppose that we are given an $n$-cycle $C \in \mathcal{C}_n$ 
and $R_1, \ldots, R_k$ are the consecutive runs.  
If $k =n$, then, clearly, $C$ is incontractible. Otherwise, 
suppose that $s$ is the least $i$ such that 
$|R_i| \geq 2$. Let $R_s = \{i,i+1, \ldots, j\}$. 
Then it is easy to see that in $C$, there are no 
pairs $\lan a, t \ran$ which are nontrivial occurrences of 
$\mu$ in $C$ for $i \leq a \leq j-1$ since $a+1$ immediately 
follows $a$ in $C$.  Similarly, there are no 
pairs $\lan s, b \ran$ which are nontrivial occurrences of 
$\mu$ in $C$ for $i+1 \leq b \leq j$ since $b-1$ immediately 
precedes $b$ in $C$. Now, suppose that $C'$ arises  
from $C$ by removing $i+1, \ldots, j$ and replacing 
each $s > j$ by $s -(j-i)$.  Then it is easy to see 
that 
\begin{enumerate} 
\item if $s < t \leq i$, $\lan s,t \ran$ is a nontrivial occurrence of 
$\mu $ in $C$ if and only if  $\lan s,t \ran$ is a nontrivial occurrence of 
$\mu $ in $C'$,

\item if $s \leq i < j < t$, then $\lan s,t \ran$ is a nontrivial occurrence of 
$\mu $ in $C$ if and only if  $\lan s,t-(j-i) \ran$ is a nontrivial occurrence of 
$\mu $ in $C'$, and 

\item if $j \leq s < t$, then $\lan s,t \ran$ is a nontrivial occurrence of 
$\mu $ in $C$ if and only if \\
 $\lan s-(j-i),t-(j-i) \ran$ is a nontrivial occurrence of 
$\mu $ in $C'$.
\end{enumerate}
It follows that $NT_{\mu}(C) = NT_{\mu}(C')$. Moreover,  
it is easy to see that 
$\cont(C) = \cont(C')$.  By induction $NT_{\mu}(C') = NT_{\mu}(\cont(C'))$. 
Hence, $NT_{ \mu}(C) = NT_{\mu}(C') = NT_{\mu}(\cont(C')) =  NT_{\mu}(\cont(C))$. 
\end{proof}

It is easy to count the number of $n$-cycles $C \in \mathcal{C}_n$ 
for which $\cont(C) =A$. That is, we have the following theorem. 

\begin{theorem}\label{th:incont}
Let $A \in \mathcal{IC}_\ell$ be an incontractible cycle of length $\ell$. The number of $n$-cycles
$C \in \mathcal{C}_n$ such that $\cont(C)=A$ is $n-1 \choose \ell-1$.
\end{theorem} 
\begin{proof}
If $\cont(C)=A$, then there are $\ell$ consecutive runs of $C$, namely 
$R_1,\dots,R_\ell$ such that $\red(\max(R_1),\dots,\max(R_\ell))=A$.
Since the maximum of one of the consecutive runs must be $n$, there are
$n-1\choose\ell-1$ ways to choose the rest of the maxima. This will determine
the consecutive runs and then we order them so that $\red(\max(R_1),\dots,
\max(R_\ell))=A$.

For example, suppose $A=(1,3,5,2,4)$ and  $n=8$. The $8$-cycle that contracts to $A$
corresponding to the choice $\{1,3,4,7\}$ out of the $7\choose4$ choices is obtained by
 letting the maxima of the
consecutive runs be $\{1,3,4,7,8\}$. Therefore, the 
consecutive runs are $\{1\},\{2,3\},\{4\},\{5,6,7\},\{8\}$. Then 
we arrange them so that
$\red(\max(R_1),\dots,\max(R_\ell))=A$ and you get $\{1\},\{4\},\{8\},\{2,3\},\{5,6,7\}$, and so the cycle
is $(1,4,8,2,3,5,6,7)$.
\end{proof}
 
Theorems \ref {ntmu} and \ref{th:incont} imply the following 
theorem. 
\begin{theorem}\label{thm:computeC} 
For any $k \geq 0$ and $n \geq 1$, 
\begin{equation}\label{eq:computeC} 
NT_{n,\mu}(q)|_{q^k}=\sum_{s=1}^n \sum_{\substack{A \in \mathcal{IC}_s,\\NT_\mu(A)=k}}
{n-1\choose s-1}.  
\end{equation} 
\end{theorem}
\begin{proof}
It is easy to see that (\ref{eq:computeC}) follows by 
partitioning the $n$-cycles $C \in \mathcal{C}_n$ such that 
$NT_{\mu}(C) =k$ by their contractions.
\end{proof}

We shall prove in the next section that 
the smallest power of $q$ that occurs in 
either $NTI_{2n,\mu}(q)$ or $NTI_{2n+1,\mu}(q)$ is $q^n$. Hence it follows 
that 
$$\sum_{s=1}^n \sum_{\substack{A \in \mathcal{IC}_s,\\NT_\mu(A)=k}}
{n-1\choose s-1} = \sum_{s=1}^{2k+1} \sum_{\substack{A \in \mathcal{IC}_s,\\NT_\mu(A)=k}}
{n-1\choose { s-1}}.$$ 

Furthermore, the maximum number of nontrivial occurrences of $\mu$ that one can 
have for a cycle of length $s$ is ${s-1 \choose 2}$, (see page \pageref{maxcycle}). This means that 
for a cycle of length $s$ to have $k$ nontrivial occurrences, we must 
have that $\frac{(s-1)(s-2)}{2} \geq k$ or, equivalently, 
$s^2 -3s+2-2k \geq 0$.  It follows that 
it must be the case that
$s\geq\frac{3+\sqrt{1+8k}}{2}$. Hence, 
\begin{equation}
\sum_{ s=1}^{2k+1} \sum_{\substack{A \in \mathcal{IC}_s,\\NT_\mu(A)=k}}
{n-1\choose  s-1}=\sum_{ s=\left\lfloor\frac{3+\sqrt{1+8k}}{2}\right\rfloor}^{2k+1} \sum_{\substack{A \in \mathcal{IC}_s,\\NT_\mu(A)=k}}
{n-1\choose  s-1}.
\end{equation}
 
Thus we have the following theorem.  
\begin{theorem}
For any $k \geq 0$ and $n \geq 1$, 
\begin{equation}
NT_{n,\mu}(q)|_{q^k} = \sum_{s={\left\lfloor\frac{3+\sqrt{1+8k}}{2}\right\rfloor}}^{2k+1} c_s \binom{n-1}{s-1}
\end{equation}
where $c_s = NTI_{s,\mu}(q)|_{q^k}$.
\end{theorem}

It follows from our tables 
for $NTI_{n,\mu}(q)$ and the fact that $NTI_{11,\mu}(q)|_{q^5} = C_5 =42$ 
which we will prove in Corollary \ref{cor:Cat} 
that for all $n \geq 1$
\begin{eqnarray*}
NT_{n,\mu}(q)|_{q} &=& \binom{n-1}{2}, \\
NT_{n,\mu}(q)|_{q^2} &=& \binom{n-1}{3} + 2\binom{n-1}{4}, \\
NT_{n,\mu}(q)|_{q^3} &=& \binom{n-1}{3} + 3\binom{n-1}{4}  +6\binom{n-1}{5} 
+ 5\binom{n-1}{6}, \\
NT_{n,\mu}(q)|_{q^4} &=& 2\binom{n-1}{4} + 13\binom{n-1}{5}  +27\binom{n-1}{6} 
+ 29\binom{n-1}{7}+14\binom{n-1}{8},\   \mbox{and} \\  
NT_{n,\mu}(q)|_{q^5} &=& \binom{n-1}{4} + 10\binom{n-1}{5}  +51\binom{n-1}{6} 
+ 134\binom{n-1}{7}+\\
&&181\binom{n-1}{8} +130\binom{n-1}{9} + 42\binom{n-1}{10}.
\end{eqnarray*}

\section{Incontractible $n$-cycles}\label{incont}

Let $IC_n$ denote the number of incontractible $n$-cycles 
in $\mathcal{C}_n$.  Clearly, there is only 1 incontractible 3-cycle, 
namely, $(1,3,2)$ and there are only 2 incontractible 4-cycles, 
namely, $(1,4,3,2)$ and $(1,3,2,4)$. Our next theorem 
shows the numbers $IC_n$ satisfies the recursion of 
the number of derangements. 
\begin{theorem} \label{thm:ICrec}
For all $n \geq 5$, 
\begin{equation}\label{ICnrec}
{IC}_n = (n-2) {IC}_{n-1} + (n-2) {IC}_{n-2}.
\end{equation}
\end{theorem}
\begin{proof}
Let $[n] = \{1, \ldots, n\}$. 
Suppose that $n \geq 5$ and $C = (1,c_2, \ldots, c_n)$ is {an} $n$-cycle 
in $\mathcal{IC}_n$. Let $C\upharpoonright_{[n-1]} = 
(1,c_2^\prime, \ldots, c_{n-1}^\prime)$ be the cycle obtained from $C$ by 
removing $n$ from $C$. For example if $C=(1,3,6,2,4,5)$ then $C\upharpoonright_{[5]}=(1,3,2,4,5)$. We then have 
two cases depending on whether  $C\upharpoonright_{[n-1]}$ is 
incontractible or not.
 \\
\ \\
{\bf Case 1.} $ C\upharpoonright_{[n-1]} \in \mathcal{IC}_{n-1}$. \\
In this situation, it is easy to see that if 
$D = C\upharpoonright_{[n-1]}$, there are exactly $n-2$ cycles 
$C' \in \mathcal{IC}_n$ such that $D = C'\upharpoonright_{[n-1]}$.
These cycles are the result of inserting $n$ immediately 
after $i$ in $D$ for $i =1, \ldots, n-2$. That is, 
let $D^{(i)}$ be the result of inserting $n$ immediately 
after $i$ in the cycle structure of $D$ where $1 \leq i \leq n-2$. 
Then it is easy to see that $D^{(i)} \in \mathcal{IC}_n$ and 
$D = D^{(i)}\upharpoonright_{[n-1]}$. 
For example, if $D = (1,3,5,4,2)$ is the 5-cycle pictured 
at the top of Figure \ref{fig:insertn}, then 
$D^{(1)}, \ldots, D^{(4)}$ are pictured on the second row 
of Figure \ref{fig:insertn}, reading from left to right. 
Thus, there are $(n-2)$ ${{IC}}_{n-1}$ 
$n$-cycles $C \in \mathcal{IC}_n$ such that $C\upharpoonright_{[n-1]}$ is an 
element of $\mathcal{IC}_{n-1}$. (Note that $D^{(n-1)}$ is not incontractible 
because it is obtained by inserting $n$ directly after $n-1$.)

\fig{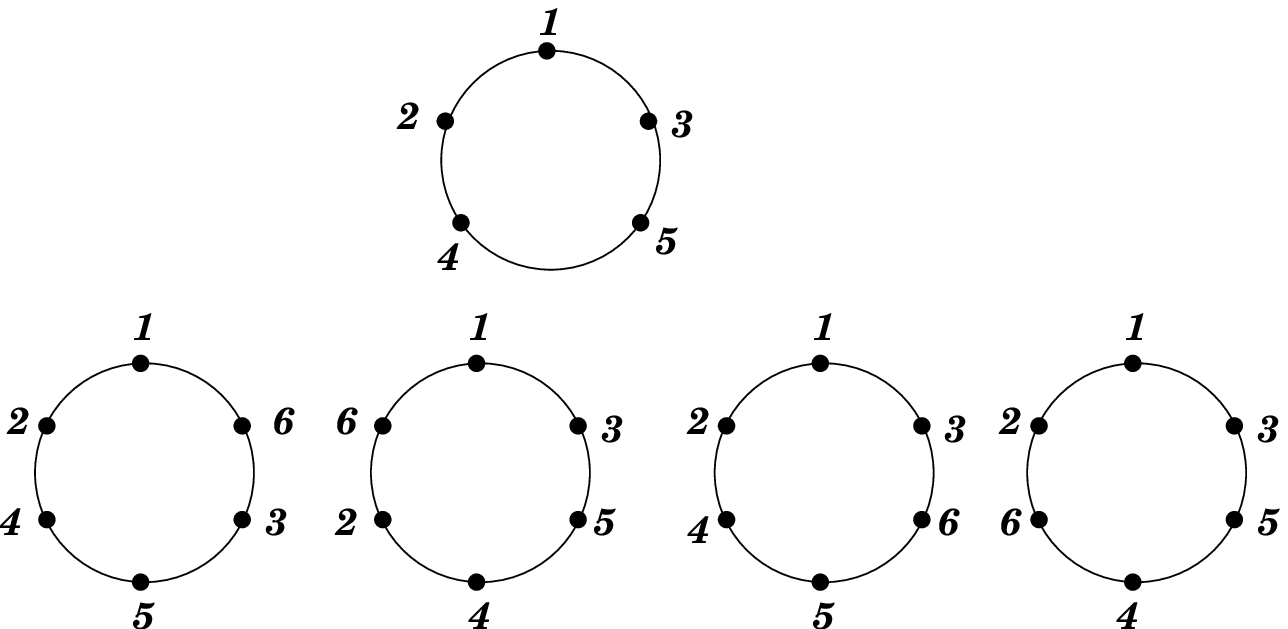}{The four elements of $\mathcal{IC}_6$ that 
arise by inserting $6$ into $C =(1,3,5,4,2)$.}

We note that if $D =(1=d_1, \ldots, d_{n-1})$ is an ${(}n-1{)}$-cycle 
in $\mathcal{IC}_{n-1}$, then 
$$NT_{\mu}(D^{(i)}) \geq 1 + NT_{\mu}(D) \ \mbox{for} \ i =1, \ldots, n-2.$$
That is, {if} $\langle d_s, d_t \rangle$ is a nontrivial occurrence of $\mu$ 
in 
$D$, then the insertion of $n$ does not effect whether 
$\langle d_s, d_t \rangle$ is a nontrivial occurrence of $\mu$ in $D^{(i)}$. 
Moreover, we are always guaranteed that $\langle i, n \rangle$ 
is a nontrivial  occurrence of $\mu$ in $D^{(i)}$ since $i \leq n-2$. 
It is possible that $NT_{\mu}(D^{(i)}) -NT_{\mu}(D)$ is  greater 
than or equal to 1.  For example, in Figure \ref{fig:insertn}, 
it is easy to see that if $D = (1,3,5,4,2)$, then there is 
only two pairs which are nontrivial occurrences of $\mu$ in $C$, namely $\langle 1,3\rangle$ and $\langle 3, 5 \rangle$, 
while in $D^{(1)}$, the insertion of $6$ after 
1 created three new nontrivial occurrences of $\mu$  in $D^{(1)}$ namely, 
the pairs 
$\langle 1,6 \rangle$, $ \langle 2,6 \rangle$, and 
$ \langle 4,6 \rangle$. Thus, $NT_{\mu}(D^{(1)}) - NT_{\mu}(D) = 5-2 =3$. 
On the other hand, it is easy to see that if 
$D \in \mathcal{IC}_{n-1}$, then 
\begin{equation}\label{insertnn-2}
NT_{\mu}(D^{(n-2)}) = 1 + NT_{\mu}(D).
\end{equation}
That is, 
if we insert $n$ immediately after $n-2$ in $D$, then for $i =1, \ldots, 
n-3$, 
$\langle i, n \rangle$ cannot be a nontrivial {occurrence} of $\mu$ in 
$D^{(n-2)}$ because 
$n-2$ will be between $i$ and $n$ in $D$. Thus, we will create 
exactly one more pair which is a 
nontrivial occurrence of $\mu$ in $D^{(n-2)}$ that was not 
in $D$, namely, $\langle n-2,n\rangle$. \\
\ \\
{\bf Case 2} $C\upharpoonright_{[n-1]} \not \in \mathcal{IC}_n$. \\
In this case, it must be that in $C$, $n$ is the only element 
between $i$ and $i+1$ in $C$ for some $i$ so that 
in $C\upharpoonright_{[n-1]}$, $i+1$ immediately follows $i$. Clearly, 
$i$ is the only $j$ in $C\upharpoonright_{[n-1]}$ such that
$j+1$ immediately follows $j$. 
Hence, we can construct an element 
of $D \in \mathcal{IC}_{n-2}$ from  $C\upharpoonright_{[n-1]}$ by removing 
$i+1$ and then replacing $j$ by $j-1$ for $i+1 <  j \leq n-1$. 
Vice versa, given $D \in \mathcal{IC}_{n-2}$ and $1 \leq i \leq n-2$, 
let $D^{[i]}$ be the cycle that results from $D$ by first replacing 
elements $j \geq i+1$, by $j+1$, then replacing $i$ by a pair 
$i$ immediately followed by $i+1$, and finally inserting 
$n$ between $i$ and $i+1$.  This process is pictured in 
Figure \ref{fig:insertn-2} for the cycle $D = (1,3,2,4)$.  
In such a situation, we shall say that $D^{[i]}$ arises by 
expanding $D$ at $i$.

\fig{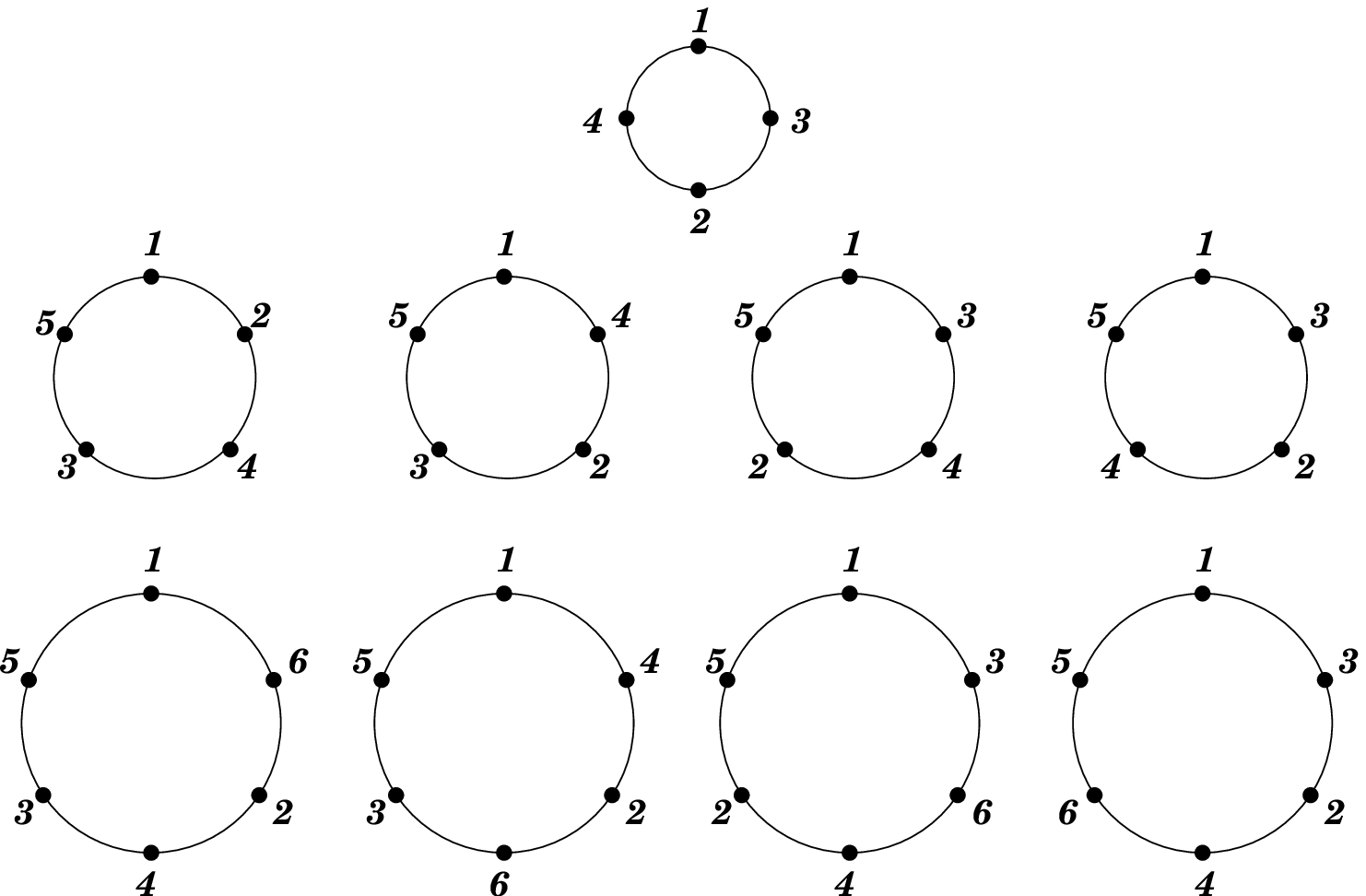}{The four elements of $\mathcal{IC}_6$ that 
arise by expanding $D =(1,3,2,4)$.}

We should note that if $D =(1=d_1, \ldots, d_{n-2})$ is an $(n-2)$-cycle 
in $\mathcal{IC}_{n-2}$, then 
$$NT_{\mu}(D^{[i]}) \geq 1 + NT_{\mu}(D) \ \mbox{for} \ i =1, \ldots, n-2.$$
In the first step of the expansion of $D$ at $i$, we replace each 
$j \geq i+1$ by $j+1$ and then replace $i$ by a consecutive pair $i$ followed 
by $i+1$ to get a cycle $D_i$. It is clear that this keeps the 
number of nontrivial occurrences of $\mu$ 
 the same.  That is, it is easy to see 
that there will be  no pair $\langle i, t\rangle$ which is a 
nontrivial occurrence of $\mu$ in $D_i$ since $i$ is immediately followed 
by $i+1$ in $D_i$. Moreover, it is easy to check that 
\begin{enumerate}
\item if $s < t \leq i$, then $\lan s, t \ran$ is a nontrivial 
occurrence of $\mu$  in $D$ if and only if 
$\lan s, t \ran$ is a nontrivial 
occurrence of $\mu$  in $D_i$,

\item if $s < i < t$, then 
$\langle s,t \rangle$ is a nontrivial 
occurrence of $\mu$  in  $D$ if and only if $\langle s,t+1 \rangle$ 
is a nontrivial 
occurrence of $\mu$  in $D_i$.

\item $\langle i,t \rangle$ is a nontrivial 
occurrence of $\mu$  
in $D$ if and only if $\langle i+1,t+1 \rangle$ is a nontrivial 
occurrence of $\mu$  in  $D_i$, 
and 

\item if $i < s < t$, then $\langle s,t \rangle$ is a nontrivial 
occurrence of $\mu$  in $D$ if and only if $\langle s+1,t+1 \rangle$ 
is a nontrivial 
occurrence of $\mu$  in $D_i$.
\end{enumerate}

Then for every pair $\langle r, s \rangle$ 
which is a nontrivial 
occurrence of $\mu$  in $D_i$, there exists a nontrivial occurrence of $\mu$ in $D^{[i]}$ after
inserting $n$ between $i$ and $i+1$. Finally we will create 
at least one new nontrivial 
occurrence of $\mu$  in $D$, namely $\langle i, n \rangle$.

Again it is possible that $NT_{\mu}(D^{[i]}) -NT_{\mu}(D)$ is greater 
than or equal to 1.  For example, in Figure \ref{fig:insertn-2}, 
it is easy to see that if $D = (1,3,5,4,2)$, there are
only two pairs that are nontrivial occurrences of $\mu$ in $D$, namely $\langle 1,3\rangle$ 
and $\langle 2, 4 \rangle$, which correspond to the 
pairs $\langle 1,4 \rangle$ and $\langle 3, 5 \rangle$ 
that are nontrivial occurrences of $\mu$ in $D^{[2]}$. However, the insertion of 6 between 
2 and 3 in $D_2$ created two new nontrivial occurrences of $\mu$ 
in $D^{[2]}$, 
namely, $\langle 2,6 \rangle$ and $\langle 4,6 \rangle$.
Thus, $NT_{\mu}(D^{[2]}) - NT_{\mu}(D) = 4-2 =2$. 
On the other hand, it is easy to see that if 
$D \in \mathcal{IC}_{n-2}$, then 
\begin{equation}\label{2insertnn-2}
NT_{\mu}(D^{[n-2]}) = 1 + NT_{\mu}(D).
\end{equation} 
That is, 
if we insert $n$ immediately after $n-2$ in $D_{n-2}$, then for $i =1, \ldots, 
n-3$, 
$\langle i, n \rangle$ cannot be a nontrivial occurrences of $\mu$
  in $D^{[n-2]}$ because 
$n-2$ will be between $i$ and $n$ in $D^{[n-2]}$. Thus, we will create 
exactly one occurrence of $\mu$ in $D^{[n-2]}$ that was not 
in $D_{n-2}$, namely, $\langle n-2,n\rangle$. \\

\end{proof}

We have two important corollaries to our proof of Theorem \ref{thm:ICrec}.

\begin{corollary} For all $n \geq 2$, ${{IC}}_n = D_{n-1}$ where 
$D_n$ is the number of derangements of $S_n$.
\end{corollary}
\begin{proof} It is well-known that $D_1 =0 = {IC}_2$, 
$D_2 = 1 = {IC}_3$, and that for $n \geq 2$, \\
$D_{n+1} = n D_{n-1}+ nD_{n-2}$. Thus, the corollary follows 
by recursion. 
\end{proof}

\begin{corollary} For all $n \geq 2$, the lowest power 
of $q$ appearing in $NTI_{2n,\mu}(q)$ and \\ 
$NTI_{2n+1,\mu}(q)$ is $q^n$. 
\end{corollary}
\begin{proof}
We have shown by direct calculation that the lowest power of 
$q$ appearing in $NTI_{4,\mu}(q)$ and $NTI_{5,\mu}(q)$ is $q^2$. Thus, 
the corollary holds for $n =2$. Now, assume the corollary holds 
for $n \geq 2$.  Then we shall show the corollary holds 
for $n+1$.  We have shown that  each $(2n+2)$-cycle $C$ in 
$\mathcal{IC}_{2n+2}$ is either of the form 
$D^{(i)}$ for some $D \in \mathcal{IC}_{2n+1}$ and $1 \leq i \leq 2n$ in which 
case $NT_{\mu}(D^{(i)}) \geq 1 + NT_{\mu}(D)$, or of the form 
$E^{[i]}$ for some $E \in \mathcal{IC}_{2n}$ and $1 \leq i \leq 2n$ in which 
case $NT_{\mu}(E^{[i]}) \geq 1 + NT_{\mu}(E)$. 
But by induction, we know that $NT_{\mu}(D) \geq n$ and 
$NT_{\mu}(E) \geq n$ so that $NT_{\mu}(C) \geq n+1$. Thus, the smallest 
possible power of $q$ that can appear in $NTI_{2n+2,\mu}(q)$ is $n+1$. 
On the other hand, we can assume by induction that there is 
a $D \in \mathcal{IC}_{2n+1}$ such that 
$NT_{\mu}(D) = n$ in which case we know 
that $NT_{\mu}(D^{(n-2)}) = 1 + NT_{\mu}(D) =n+1$. Hence, the coefficient 
of $q^{n+1}$ in $NT_{2n+2,\mu}(q)$ is non-zero.

We have shown that  each $(2n+3)$-cycle $C$ in 
$\mathcal{IC}_{2n+3}$ is either of the form 
$D^{(i)}$ for some $D \in \mathcal{IC}_{2n+2}$ and $1 \leq i \leq 2n+1$ in which case $NT_{\mu}(D^{(i)}) \geq 1 + NT_{\mu}(D)$, or of the form 
$E^{[i]}$ for some $E \in \mathcal{IC}_{2n+1}$ and $1 \leq i \leq 2n+1$ 
in which 
case $NT_{\mu}(E^{[i]}) \geq 1 + NT_{\mu}(E)$. 
But by induction, we know that $NT_{\mu}(D) \geq n+1$ and 
$NT_{\mu}(E) \geq n$ so that $NT_{\mu}(C) \geq n+1$. Thus, the smallest 
possible power of $q$ that can appear in $NTI_{2n+3,\mu}(q)$ is $n+1$. 
On the other hand, we can assume by induction that there is 
a $D \in \mathcal{IC}_{2n+1}$ such that 
$NT_{\mu}(D) = n$ in which case we know 
that $NT_{\mu}(D^{[n-2]}) = 1 + NT_{\mu}(D) =n+1$. Hence, the coefficient 
of $q^{n+1}$ in $NT_{2n+3,\mu}(q)$ is non-zero. 
\end{proof}

We have not been able to find a recursion for the 
polynomials $NTI_{n,\mu}(q)$. The problem with the recursion 
implicit in the proof of Theorem \ref{thm:ICrec} is that
for $n$-cycles $D \in \mathcal{IC}_n$, the contributions of 
$\sum_{i\geq 1} q^{NT_{\mu}(D^{(i)})}$ are not uniform. 
For example, if $D = (1,3,5,4,2)$, then 
$NT_{\mu}(D) = 2$, $NT_{\mu}(D^{(1)}) = 5$, $NT_{\mu}(D^{(2)}) = 4$, $NT_{\mu}(D^{(3)}) = 4$, and $NT_{\mu}(D^{(4)}) = 3$ so that 
$\sum_{i=1}^4 q^{NT_{\mu}(D^{(i)})} =(q +2q +q^3)q^{NT_{\mu}(D)}$. 
However, $D = (1,5,4,3,2)$, then 
$NT_{\mu}(D) = 6$, $NT_{\mu}(D^{(1)}) = 10$, $NT_{\mu}(D^{(2)}) = 9$, $NT_{\mu}(D^{(3)}) = 8$, and $NT_{\mu}(D^{(4)}) = 7$ so that 
$\sum_{i=1}^4 q^{NT_{\mu}(D^{(i)})} =(q +q^2+q^3 +q^4)q^{NT_{\mu}(D)}$. 
A similar phenomenon occurs for $\sum_{i\geq 1} NT_{\mu}(D^{[i]})$. 
For example, it is easy to see from Figure 6 that if 
$D= (1,3,2,4)$, then $NT_{\mu}(D) = 2$, $NT_{\mu}(D^{[1]}) = 3$, 
$NT_{\mu}(D^{[2]}) = 4$,  $NT_{\mu}(D^{[3]}) = 3$,  and 
$NT_{\mu}(D^{[4]}) = 3$, so that 
$\sum_{i=1}^4 q^{NT_{\mu}(D^{[i]})} =(3q+q^2)q^{NT_{\mu}(D)}$. 
However, as one can see from Figure \ref{fig:2insertn-2} below 
that if $D=(1,4,3,2)$, then $NT_{\mu}(D) = 3$, $NT_{\mu}(D^{[1]}) = 6$, 
$NT_{\mu}(D^{[2]}) = 5$,  $NT_{\mu}(D^{[3]}) = 4$,  and 
$NT_{\mu}(D^{[4]}) = 4$, so that 
$\sum_{i=1}^4 q^{NT_{\mu}(D^{[i]})} =(2q+q^2+q^3)q^{NT_{\mu}(D)}$.

\fig{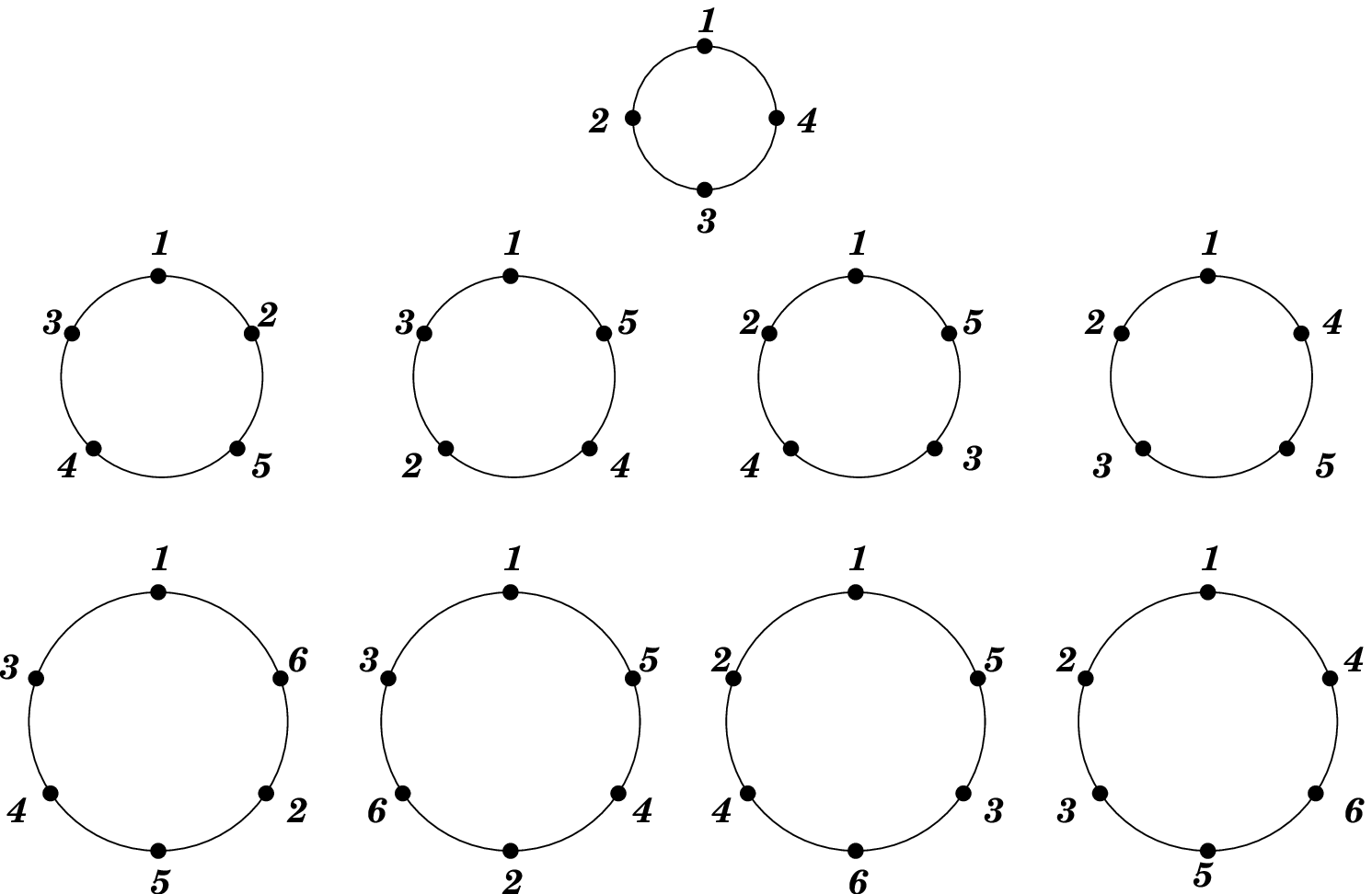}{The four elements of $\mathcal{IC}_6$ that 
arise by expanding $D =(1,4,3,2)$.}

Our next goal is to show that $NTI_{2n+1,\mu}(q)|_{q^n} = C_n$ 
where $C_n$ is the $n$th Catalan number. 

\begin{theorem} \label{thm:chpath}
Suppose that $\sg = \sg_1 \ldots \sg_{2n+1} \in 
S_{2n+1}$ and {\rm $\mbox{chpath}(\sg)=P$} is a Dyck path of length $2n$. Then  
the $n$-cycle $C_\sg = (\sg_1, \ldots, \sg_{2n+1})$  is incontractible and  
$NT_{\mu}(C) = n$. 
\end{theorem}

\begin{proof}
Since the path $P$ starts  at $(0,0)$, it follows that $\mbox{ind}_{\sg}(\sg_1) = 0$ . This can only happen if $\sg_1 =1$. 
Thus, $C_\sg =(1,\sg_2, \ldots, \sg_{2n+1})$ has the standard 
form of an $n$-cycle. Note that the only way $i+1$ immediately 
follows $i$ in $C_\sg$ is if there is a $j$ such that 
$\sg_j =i$ and $\sg_{j+1} =i+1$. But then 
$ind_\sg(\sg_j) = ind_\sg(\sg_{j+1})$ which would imply 
that the charge path of $\sg$ has a level which means that there is an edge drawn horizontally.  Since 
we are assuming that the charge path of $C_\sg$ is a Dyck path, 
there can be no such $i$ and, hence, $C_\sg$ is incontractible. 

Next we claim that 
if the $i$th step of $P$ is an up-step, then 
$\langle \sg_i, \sg_{i+1} \rangle$ is a nontrivial occurrence of 
$\mu$ in $C_{\sg}$. That is, if the 
$i$th step of $P$ is an up-step, then 
$\mbox{ind}_{\sg}(\sg_i) = k$ and $\mbox{ind}_{\sg}(\sg_{i+1}) = k+1$ 
for some $k \geq 0$. 
Since $P$ is a Dyck path, there must be some $j > i+1$ such 
that $\mbox{ind}_{\sg}(\sg_j) = k$ because we must pass through 
level $k$ from the point $(i+1,k+1)$ in $\mbox{chpath}(\sg)$ to 
get back to the point $(2n,0)$ which is the end point of 
the path. Let $i_j$ denote the smallest $j > i$ such 
that $\mbox{ind}_{\sg}(\sg_j) =k$.  It follows from 
the definition of the function $\mbox{ind}_\sg$ that 
it must be the case that $\sg_i = \ell$ and $\sg_{i_j} = \ell +1$. 
Moreover, it must be the case that $\sg_{i+1} > \ell +1$ so that 
$\langle \sg_i, \sg_{i+1} \rangle$  is a nontrivial occurrence of 
$\mu$ in $C_{\sg}$. We claim that there can be 
no $t \neq \sg_{i+1}$ such that $\lan \sg_i, t \ran$ is a nontrivial 
occurrence of $\mu$ in $C_\sg$. That is, 
since $P$ is a Dyck path, we know that all the vertices 
on $P$ between $(i,k)$ and $(i_j,k)$ lie on levels strictly 
greater than $k$. It is easy to see from our inductive definition of 
$\mbox{ind}_{\sg}$ that it must be the case that 
$\sg_{i+1}$ is the least element of $\sg_{i+1}, \sg_{i+2}, 
\ldots, \sg_{i_j-1}$ and hence the fact that $\sg_{i+1}$ immediately 
follows $\sg_i$ in $C_\sg$ implies that none of the 
pairs $\langle \sg_i, \sg_s \rangle$ 
 is a nontrivial occurrence of 
$\mu$ in $C_\sg$ for $i+1 < s < i_j$. But then as we traverse around 
the cycle $C_\sg$, the fact that $\sg_i = \ell$ and 
$\sg_{i_j} = \ell +1$ implies that none of the pairs 
$\langle \sg_i, \sg_s \ran$ where either $i_j < s \leq 2n+1$ or 
$1 \leq s < i$  are nontrivial occurrences of 
$\mu$ in $C_\sg$. Thus, 
if the $i$th step of $P$ is an up-step, then 
$\langle \sg_i, t\rangle$ is an occurrence of $\mu$ in $C_\sg$ for 
exactly one $t$, namely $t = \sg_{i+1}$. 

Next suppose that the $i$th step of $P$ is a down-step. Then, we claim
that there is no $t$ such that 
$\langle \sg_i, t\rangle$ is an occurrence of $\mu$ in $C_\sg$. 
We then have two cases. \\
\ \\
{\bf Case 1.} $\mbox{ind}_{\sg}(\sg_i) =k$ and there 
is a $j > i$ such that  $\mbox{ind}_{\sg}(\sg_j) =k$.\\
In this case, let $i_j$ be the least $j>i$ such  
that $\mbox{ind}_{\sg}(\sg_j) =k$. It then follows 
that if  $\sg_i = \ell$, then $\sg_{i_j} = \ell +1$. Because 
$P$ is a Dyck path, it must be the case that 
$\mbox{ind}_{\sg}(\sg_s) < k$ for all $i < s < i_j$ so 
that $\sg_s < \ell$ for all $i < s < i_j$. But this means 
that as we traverse around the cycle $C_\sg$, the first element 
that we encounter after $\sg_i = \ell$ which is bigger than 
$\ell$ is $\sg_{i_j} =\ell+1$ and 
hence there is no $t > \sg_i$ such that 
$\langle \sg_i,t \rangle$ is a nontrivial  occurrence of $\mu$ in $C_\sg$. \\
\ \\
{\bf Case 2.} $\mbox{ind}_{\sg}(\sg_i) =k$ and there 
is no $j > i$ such that  $\mbox{ind}_{\sg}(\sg_j) =k$.\\
In this case since $P$ is a Dyck path,  
all the elements $\sg_s$ such that $s > i$ 
must have $\mbox{ind}_{\sg}(\sg_s) < k$ and hence 
$\sg_s < \sg_i$ since, in a charge graph, all the elements whose 
index in $\sg$ is less than $k$ are smaller than all the elements 
whose index is equal to $k$ in $\sg$.  We now have two subcases. \\
\ \\
{\bf Subcase 2.1.} There is an $s < i$ such that $\mbox{ind}_{\sg}(\sg_s) = k+1$. \\
Let $s_j$ be the smallest $s < i$ such that $\mbox{ind}_{\sg}(\sg_s) = k+1$. Since $\sg_i$ was the rightmost 
element whose index relative to $\sg$ equals $k$, it follows 
that if $\sg_i =\ell$, then $\sg_{s_j} = \ell+1$ since it is the leftmost
 element  whose index relative to $\sg$ equals $k+1$. 
Moreover, for all $t < s_j$, either $\sg_t$ has index less than 
$k$ in $\sg$ or $\sg_t$ has index $k$ in $\sg$. In either case, 
our definition of $\mbox{ind}_{\sg}$ ensures that 
$\sg_t < \ell$. But this means 
that as we traverse around the cycle $C_\sg$, the first element 
that we encounter after $\sg_i = \ell$ which is bigger than 
$\ell$ is $\sg_{s_j} = \ell+1$ and 
hence there is no $t > \sg_i$ such that 
$\langle \sg_i,t \rangle$ is a nontrivial occurrence of $\mu$ in $C_\sg$. \\
\ \\
{\bf Subcase 2.2.} There is no $s < i$ such that 
$\mbox{ind}_{\sg}(\sg_s) = k+1$.\\
In this case, $\sg_i$ has the highest possible index in $\sg$ and it 
is the rightmost element whose  index in $\sg$ is $k$ which means 
that $\sg_i =2n+1$. Hence, there is no $t > \sg_i$ such that 
$\langle \sg_i,t \rangle$ is an occurrence of $\mu$ in $C_\sg$.\\
\ \\
The only other possibility is that we have not considered is that 
$i =2n+1$.  Since $P$ is a Dyck path, we know that 
$\mbox{ind}_{\sg}(\sg_{2n+1}) =0$, $\mbox{ind}_{\sg}(\sg_{1}) =0$,
and $\mbox{ind}_{\sg}(\sg_{2}) =1$. But then it follows 
that if $\sg_{2n+1} = \ell$, then $\sg_1 < \ell$ and 
$\sg_2 = \ell+1$. But this means 
that as we traverse around the cycle $C_\sg$, the first element 
that we encounter after $\sg_{2n+1} = \ell$ which is bigger than 
$\ell$ is $\sg_{2} = \ell+1$ and 
hence there is no $t > \sg_i$ such that 
$\langle \sg_{2n+1},t \rangle$ is a nontrivial occurrence of $\mu$ in $C_\sg$.

Thus, we have shown that the nontrivial occurrences of $\mu$ in 
 $C_\sg$ correspond 
to the up-steps of $\mbox{chpath}(\sg)$ and hence 
$NT_{\mu}(C_\sg) =n$. 
\end{proof}

Our next theorem will show that if $C = (1,\sg_2, \ldots, \sg_{2n+1})$ is a
$(2n+1)$-cycle in $\mathcal{IC}_{2n+1}$ where 
$NT_{\mu}(C) =n$, then the charge path 
of $\sg_C = 1 \sg_2 \ldots \sg_{2n+1}$ must be a Dyck path.

\begin{theorem} \label{thm:minchpath}
Let  $C = (1,\sg_2, \ldots, \sg_{2n+1}) \in \mathcal{IC}_{2n+1}$ and 
$\sg_C = 1 \sg_2 \ldots \sg_{2n+1}$. Then, if $NT_{\mu}(C) =n$, 
the charge path of $\sg_C$ must be a Dyck path.
\end{theorem}
\begin{proof}
Our proof proceeds by induction on $n$. 
For $n =1$, the only element of $\mathcal{IC}_{3}$ is $C =(1,3,2)$ and 
it is easy to see that $\mbox{chpath}(132)$ is a Dyck path. 

Now assume by induction that if $D = (1,\sg_2, \ldots, \sg_{2n+1})$ is a
$(2n+1)$-cycle in $\mathcal{IC}_{2n+1}$ such that $NT_{\mu}(D) =n$, 
then  the charge path of 
$\sg_D = 1 \sg_2 \ldots \sg_{2n+1}$ is a Dyck path. 
Now suppose that  $C = (1, \sg_2, \ldots, \sg_{2n+3})$ is a $(2n+3)$-cycle in $\mathcal{IC}_{2n+3}$ such that $NT_{\mu}(C) =n+1$. Let 
$\sg_C = \sg = 1 \sg_2 \ldots \sg_{2n+3}$. Then, we know 
by our proof of Theorem \ref{thm:ICrec} that either 
$C =D^{(j)}$ for some $D \in \mathcal{IC}_{2n+2}$ or $C =D^{[j]}$ for 
some $D \in \mathcal{IC}_{2n+1}$. We claim that  $C$ cannot be 
of the form $D^{(j)}$ for some $D \in \mathcal{IC}_{2n+2}$. 
That is, we know that 
for all $D \in \mathcal{IC}_{2n+2}$ we have $NT_\mu (D) \geq n+1$ and, hence, 
$NT_\mu (D^{(j)}) \geq 1+ NT_\mu (D) \geq n+2$. 
Thus, there must be some $D \in \mathcal{IC}_{2n+1}$ such that 
$C = D^{[j]}$ for some $j$.  But then we know that 
$P_D = \mbox{chpath}(D)$ is a Dyck path. 
Let $D =(1,\tau_2, \ldots, \tau_{2n+1})$ and 
$\sg_D = \tau = \tau_1 \ldots \tau_{2n+1}$. 
Note that it follows from our arguments in 
Theorem \ref{thm:ICrec} that 
$NT_{\mu}(D^{[j]}) - NT_{\mu}(D)$ is equal to the number of 
pairs $\langle s, 2n+3 \rangle$ which are nontrivial occurrences of 
$\mu$ in  $D^{[j]}$. It is also the case that the 
charge path of $\sg$ is easily constructed from the charge path 
of $\tau$. That is, suppose that $\tau_i =j$. Then 
$1, \ldots, j$ are in the same order in 
$\sg$ and $\tau$ so that $\ind_{\sg}(s) = \ind_{\tau}(s)$ 
for $1 \leq s \leq j$.  In going from $\tau$ to $\sg$, we 
first increased  the value of any $t > j$ by one and then 
inserted $j+1$ immediately after $j$ to get a permutation $\alpha$. 
 Thus 
$\ind_{\alpha}(j) = \ind_{\alpha}(j+1)$. Moreover,  
$j+1, \ldots, 2n+2$ are in the same relative order in $\alpha$  as 
$j, \ldots, 2n+1$ in $\tau$ so that 
for $j+1 \leq t \leq 2n+1$, $\ind_{\tau}(t) = \ind_{\alpha}(t+1)$. 
Finally, inserting $2n+3$ in $\alpha$ between $j$ and $j+1$ has 
no effect on the indices assigned to $1, \ldots , 2n+2$. 
Thus for $1 \leq s \leq j$, $\ind_{\sg}(j) = \ind_{\tau}(j)$ 
and for $j+1 \leq t \leq 2n+1$, $\ind_{\sg}(t+1) = \ind_{\tau}(t)$.
This means that one can construct the charge path of 
$\sg$ by essentially starting with the charge graph of 
$\tau$, then replacing the vertex $(i, \ind_{\tau}(\tau_i))$ by a horizontal 
edge, and finally replacing that horizontal edge 
by a pair of edges $\{(i,\ind_{\sg}(\sg_i)), (i+1, \ind_{\sg}(2n+3)\}$ 
and $\{(i+1, \ind_{\sg}(2n+3)),((i+2),\ind_{\sg}(\sg_{i+2})\}$. 
In particular, this means the charge paths from $1$ up to $i$ 
are identical for both $\sg$ and $\tau$ and the charge 
path from $i+2$ to $2n+3$ in $\sg$ is identical to the 
charge path from $i$ to $2n+1$ in $\tau$. 

We now consider several cases depending on 
which $i$ is such that $\tau_i =j$ and the value 
$\mbox{ind}_{\tau}(\tau_i)$ in the Dyck path $P_D$.

\ \\
{\bf Case 1.} $j = \tau_i$ where  $\mbox{ind}_{\tau}(\tau_{i}) 
< \mbox{ind}_{\tau}(\tau_{i+1})$.\\
In this case, we will have $\sg_i = \tau_i$, $\sg_{i+1} = 2n+3$, 
$\sg_{i+2} = 1 + \tau_i$, and $\sg_{i+3} = 1 + \tau_{i+1}$. 
Since $P_D$ is a Dyck path, there will be some 
$k$ such that 
$\mbox{ind}_{\tau}(\tau_i) =k$, and $\mbox{ind}_{\tau}(\tau_{i+1}) =k+1$ so 
that we will be in the situation pictured in Figure \ref{fig:Case1A}.
Because $\mbox{ind}_{\sg}(\sg_{i+3}) =k+1$, it must be the case that 
$\mbox{ind}_{\sg}(2n+3) >k+1$. It is easy to see that 
$\langle \sg_i, 2n+3\rangle$ is an occurrence of $\mu$ in $C$. 
We then have two subcases. \\
\ \\
{\bf Subcase 1A.} There is an $s < i$ such that 
$\mbox{ind}_{\sg}(\sg_s) =k+1$.\\
In this case, let $s_i$ be the largest $s$ such that $s < i$ and  
$\mbox{ind}_{\sg}(\sg_s) =k+1$. Because $P_D$ is a Dyck path, it must be 
the case that  for all $s_i < t < i$, $\mbox{ind}_{\sg}(\sg_{s_i}) \leq k$ because 
on the path from $(s,k+1)$ to $(i,k)$, we cannot 
have an $s < t < i$ such that $\ind_{\sg}(\sg_t) > k+1$ without 
having some $t < u < i$ such that $\ind_{\sg}(\sg_u) = k+1$ 
which would violate our choice of $s_i$. 
But this means that $\sg_t < \sg_i$ for all $s_i < t < \sg_i$ and 
hence, $\langle \sg_{s_i},2n+3\rangle$ is a nontrivial occurrence of 
$\mu$ in $C$. Thus, $NT_{\mu}(C) \geq NT_{\mu}(D)+2 = n+2$.\\
\ \\
{\bf Subcase 1B.} There is no $s < i$ such that 
$\mbox{ind}_{\sg}(\sg_s) =k+1$.\\
In this case, let $s_i$ be the largest $s$ such that $i < s$ and  
$\mbox{ind}_{\sg}(\sg_s) =k$. Because $P_D$ is a Dyck path such an 
$s$ must exist since $\mbox{ind}_{\sg}(\sg_{i+3}) =k+1$ and
a Dyck path that reaches level $k+1$ 
 must subsequently descend to level $k$. 
But this means that for all $t < i$, $\mbox{ind}_{\sg}(\sg_{t}) \leq k$ and 
hence, $\sg_t \leq \sg_i$. Moreover, for all $r > s_i$,  
$\mbox{ind}_{\sg}(\sg_{r}) < k$ and 
hence, $\sg_r \leq \sg_i$. But it then follows that since 
$\sg_{s_i} > \sg_i$, $\langle \sg_{s_i},2n+3\rangle$ is a nontrivial occurrence of 
$\mu$ in $C$. Thus, $NT_{\mu}(C) \geq  NT_{\mu}(D) + 2 = n+2$.

\fig{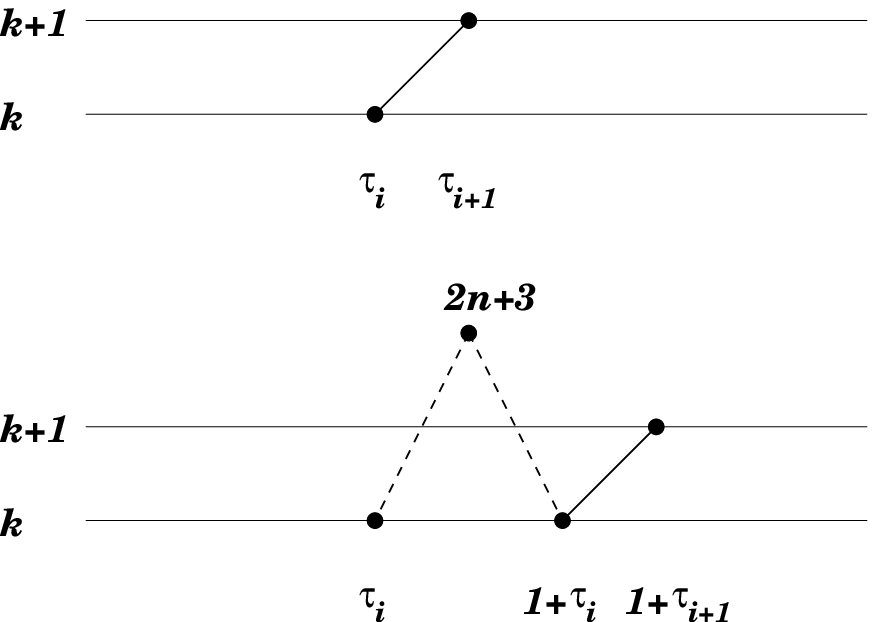}{The charge path of $D^{[\tau_{2n+1}]}$ in Case 1.}

\noindent
{\bf Case 2.} $j = \tau_{2n+1}$. \\
In this case, we will have $\sg_{2n+1} = \tau_{2n+1}$, $\sg_{2n+2} = 2n+3$ and 
$\sg_{2n+3} = \tau_{2n+1}+1$. Thus, 
the charge path of $D^{[\tau_{2n+1}]}$ looks 
like one of the two situations pictured in Figure \ref{fig:Case2}. 
That is, we have two subcases. \\
\ \\
{\bf Subcase 2A.} $\mbox{ind}_{\sg}(\sg_{2n+2}) = 1$.\\
In this case, $\mbox{chpath}( \sg)$ will be a Dyck path. This case 
can only happen if $\mbox{ind}_{\tau}(\tau_i) \leq 1$ for all $
1 \leq i \leq 2n+1$. \\
\ \\
{\bf Subcase 2B.} $\mbox{ind}_{ \sg}(\sg_{2n+2}) \geq  2$.\\
In this case, $\mbox{chpath}(\sg)$ will not be a Dyck path. This case 
can only happen if there is an $s$ such that
$\mbox{ind}_{\tau}(\tau_s) = \ind_{\sg}(\tau_i+s) \geq 2$. But then since 
$\mbox{ind}_{\sg}(\sg_{2n}) = 1$, it must be the case that 
$\sg_{2n} < \tau_s +1 < 2n+3$. 
Hence both $\langle \sg_{2n+1}, 2n+3\rangle$ and 
$\langle \sg_{2n}, 2n+3\rangle$ are nontrivial occurrences of 
$\mu$ in $C$ so that $NT_{\mu}(C) \geq  NT_{\mu}(D)+2 = n+2$. 

\fig{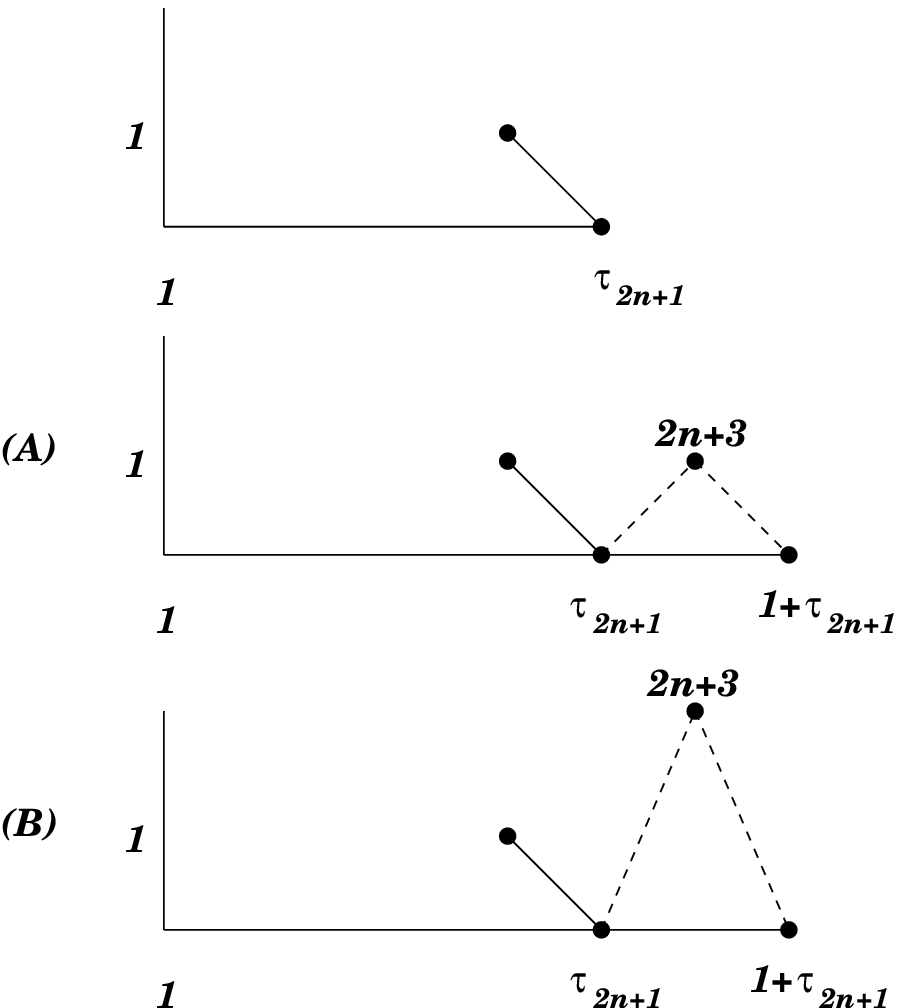}{The charge path of $D^{[\tau_{2n+1}]}$ in Case 2.}

\ \\
{\bf Case 3.} $j = \tau_i$ where 
$\mbox{ind}_{\tau}(\tau_{i-1}) > \mbox{ind}_{\tau}(\tau_{i}) 
> \mbox{ind}_{\tau}(\tau_{i+1})$.\\
In this case, we will have 
$\sg_{i-1} = 1+\tau_{i-1}$, $\sg_i = \tau_i$, $\sg_{i+1} = 2n+3$, 
$\sg_{i+2} = 1 + \tau_i$, and $\sg_{i+3} = \tau_{i+1}$. There 
are now two cases.  \\
\ \\
{\bf Case 3A.} $\tau_{i-1} < 2n+1$. \\
In this case, we will have the situation pictured in 
Figure \ref{fig:Case4A}.  That is, since $\sg_{i-1} = 1+ \tau_{i-1} < 2n+2$, 
it must be the case the either $2n+2$ is to the left of $\sg_{i-1}$ in 
which case $\ind_{\sg}(2n+2) > k+1$ or $2n+2$ is the right of $\sg_{i+3}$ 
in which case $\ind_{\sg}(2n+2) \geq k+1$ and $2n+3$ is to the left 
of $2n+2$ in $\sg$. In either case, it 
must be that $\mbox{ind}_{\sg}(2n+3) > k+1$.  It is then easy to 
see that both $\langle \sg_{i-1},2n+3 \rangle$ and  
$\langle \sg_{i},2n+3 \rangle$ are nontrivial occurrences of $\mu$ in 
$C$ so that $NT_{\mu}(C) \geq  NT_{\mu}(D)+2 = n+2$.

\fig{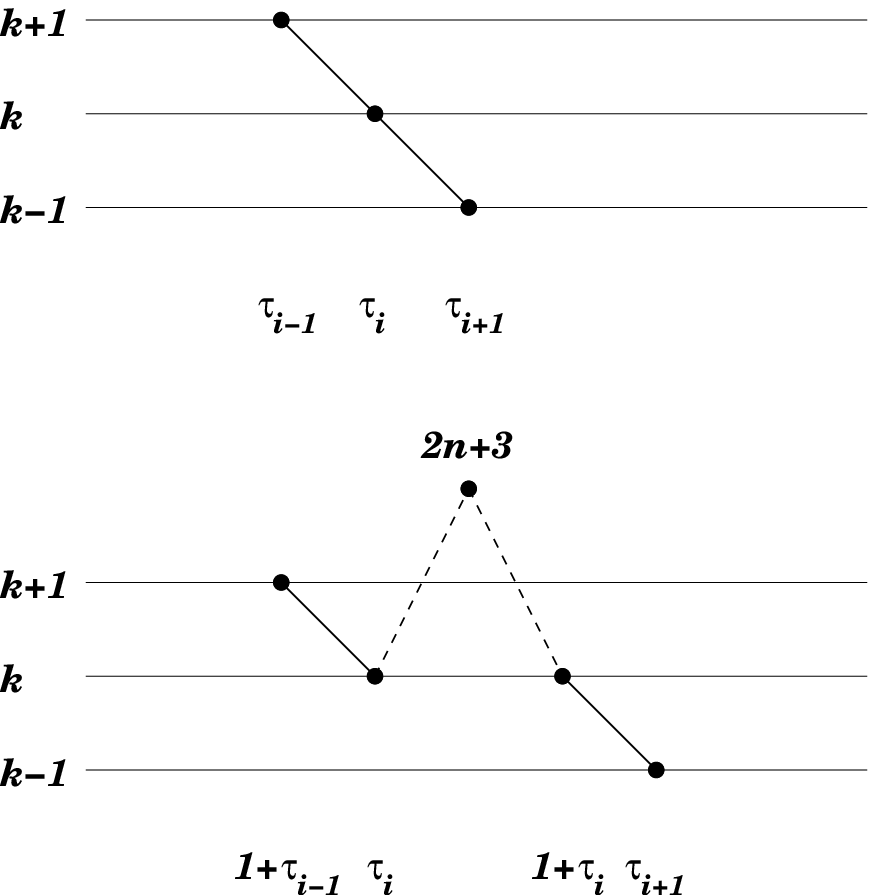}{The charge path of $D^{[\tau_i]}$ in Case 3A.}
\ \\
{\bf Case 3B.} $\tau_{i-1} =2n+1$. \\
In this case, we will have the situation pictured in 
Figure \ref{fig:Case4B}.  In this situation, 
the charge path of $D^{[\tau_i]}$ will be a Dyck path.

\fig{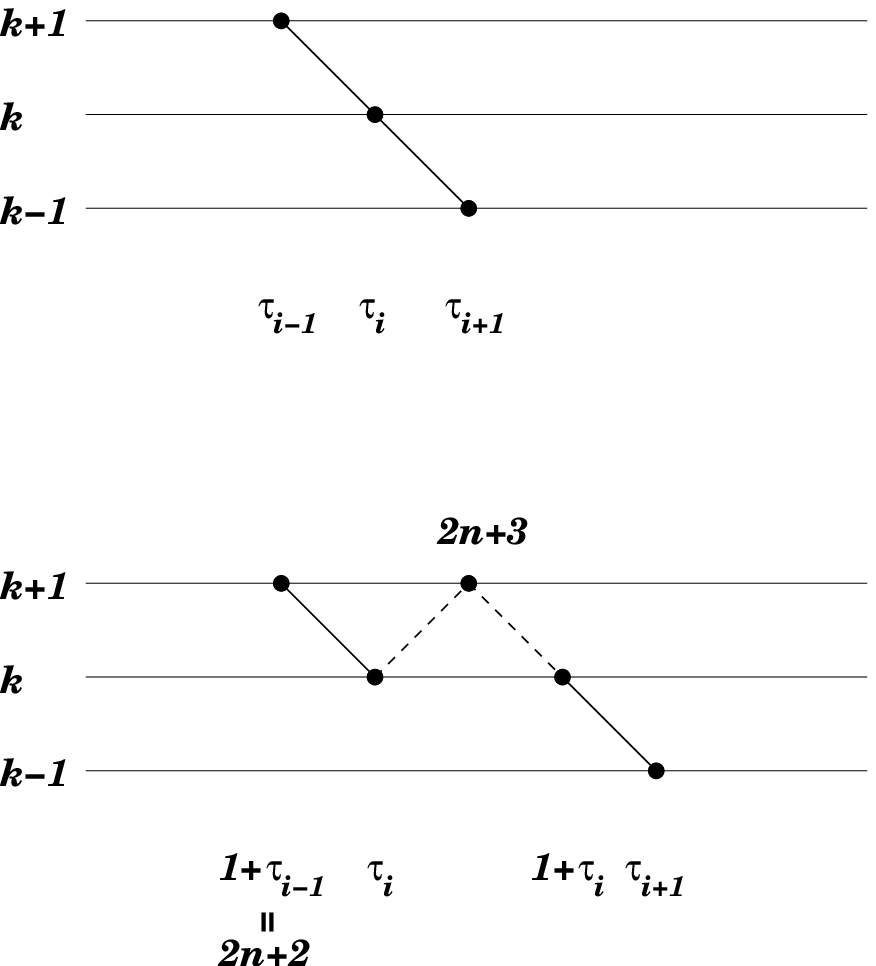}{The charge path of $D^{[\tau_i]}$ in Case 3B.}
\ \\
{\bf Case 4.} $j=\tau_i$ where 
$\mbox{ind}_{\tau}(\tau_{i-1}) < \mbox{ind}_{\tau}(\tau_{i}) 
 > \mbox{ind}_{\tau}(\tau_{i+1})$.\\
In this case, we will have 
$\sg_{i-1} = \tau_{i-1}$, $\sg_i = \tau_i$, $\sg_{i+1} = 2n+3$, 
$\sg_{i+2} = 1 + \tau_i$, and $\sg_{i+3} = \tau_{i+1}$. There 
are now two subcases. \\
\ \\
{\bf Subcase 4A.} $\mbox{ind}_{\sg}(2n+3) > k+1$.\\
In this case, we have the situation pictured in Figure 
\ref{fig:Case6A}. It will always be the case 
that $\langle \sg_i, 2n+3 \rangle$ is a nontrivial occurrence of $\mu$  
in $C$. We then have two more subcases. \\ 
\ \\
{\bf Subcase 4A1.} There is an $s < i$ such that 
$\mbox{ind}_{ \sg}(\sg_s) =k+1$.\\
In this case, let $s_i$ be the largest $s$ such that $s < i$ and  
$\mbox{ind}_{\sg}(\sg_s) =k+1$. Because $P_D$ is a Dyck path, we 
can argue as in Case 1A  
that it must be the case that 
for all $s_i < t < i$, $\mbox{ind}_{\sg}(\sg_{s_t}) \leq k$. 
But this means that $\sg_t < \sg_i$ for all $s_i < t < \sg_i$ and 
hence, $\langle \sg_{s_i},2n+3\rangle$ is also a nontrivial occurrence 
of $\mu$  
in $C$. Thus, $NT_{\mu}(C) \geq NT_{\mu}(D) + 2 = n+2$.\\
\ \\
{\bf Subcase 4A2.} There is no $s < i$ such that 
$\mbox{ind}_{ \sg}(\sg_s) =k+1$.\\
In this case, let $s_i$ be the largest $s$ such that $i < s$ and  
$\mbox{ind}_{ \sg}(\sg_s) =k$. Note that $s_i$ exists since 
$\mbox{ind}_{\sg}(\sg_{i+2}) =k$. 
But this means that for all $t < i$, $\mbox{ind}_{\sg}(\sg_{t}) \leq k$ 
and,  
hence, $\sg_t < \sg_i$. Moreover, for all $r > s_i$,  
$\mbox{ind}_{\sg}(\sg_{r}) < k-1$ and 
hence, $\sg_r < \sg_i$. But it then follows that since 
$\sg_{s_i} > \sg_i$, $\langle \sg_{s_i},2n+3\rangle$ is also a nontrivial occurrence of $\mu$   
in $C$. Thus, $NT_{\mu}(C) \geq  NT_{\mu}(D) + 2 = n+2$.

\fig{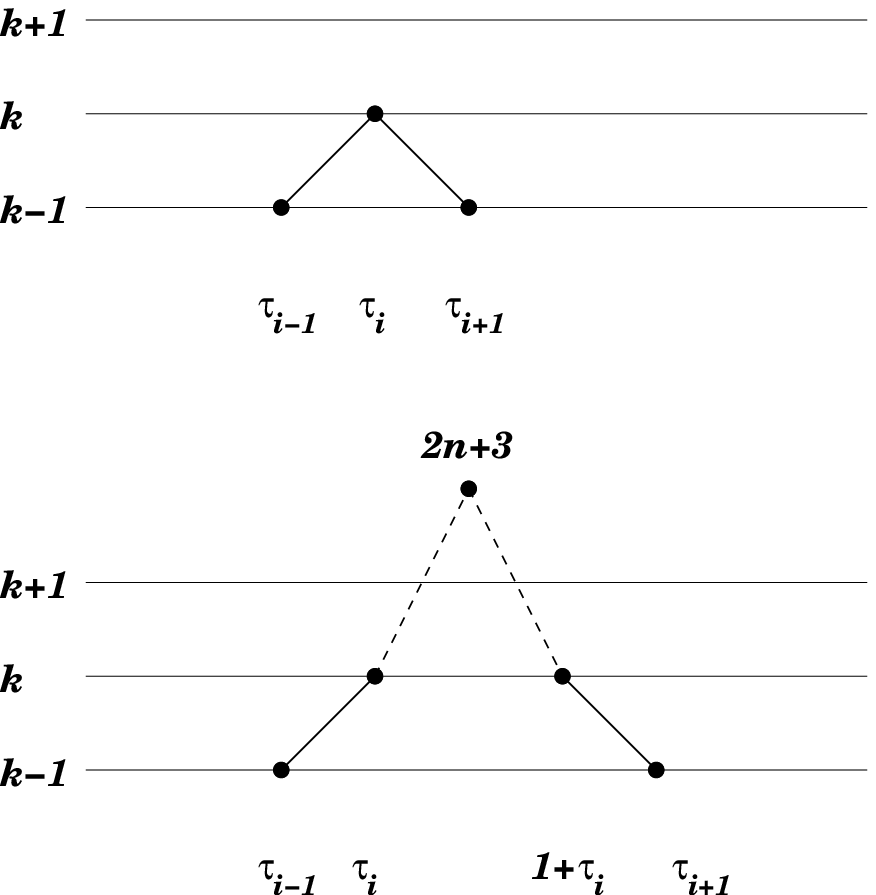}{The charge path of $D^{[\tau_i]}$ in Case 4A.}
\ \\
{\bf Subcase 4B.} $\mbox{ind}_{\sg}(2n+3) = k+1$.\\
In this case, we have the situation pictured in Figure 
\ref{fig:Case6B} and, hence, $\mbox{chpath}(\sg)$ will be a Dyck path. 

\fig{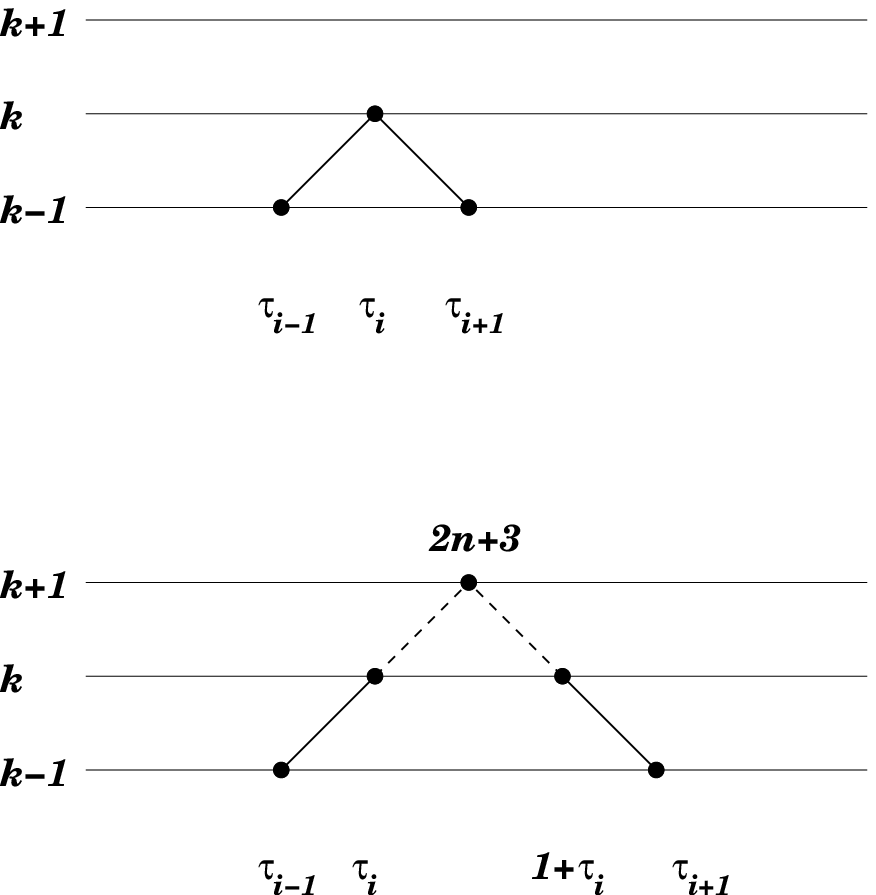}{The charge path of $D^{[\tau_i]}$ in Case 4B.}

Thus, we have shown that if $D \in \mathcal{IC}_{2n+1}$ is such that 
$NT_{\mu}(D) =n$ and $D^{[i]} =C$, then either 
$\mbox{chpath}(\sg)$ is a Dyck path or $NT_{\mu}(D^{[i]}) \geq n+2$. 
Hence, if $NT_{\mu}(C) =n+1$, then $\mbox{chpath}(\sg)$ is a Dyck path. 
\end{proof}

Theorems \ref{thm:chpath} and \ref{thm:minchpath} yield 
the following corollary.

\begin{corollary}\label{cor:Cat} 
$NTI_{2n+1,\mu}(q)|_{q^n} = C_n$ where $C_n = \frac{1}{n+1}\binom{2n}{n}$ 
is the $n$th Catalan number. 
\end{corollary}

\section{Plots and Integer Partitions}\label{intpart}

 In this section we will show that $NTI_{n,\mu}(q)|_{q^{\binom{n-1}{2} -k}}$
  and $NT_{n,\mu}(q)|_{q^{\binom{n-1}{2} -k}}$ are equal to the number of 
partitions of $k$ for sufficiently large $n$ by plotting the \emph{non-$\mu$-matches} on a grid
and mapping them to integer partitions.

Given an $n$-cycle $C \in \mathcal{C}_n$, we say that a pair 
$\langle i,j\rangle $ with $i<j$ is a non-$\mu$-match of 
$C$ if $\langle i,j\rangle $ is not an occurrence of $\mu$ in $C$.
 In other words, $\langle i,j\rangle $ is a non-$\mu$-match if there exists an integer $x$ such that $i < x < j$ and $x$ is cyclically between $i$ and $j$ in $C$.
 Furthermore, let $\mathcal{NM}_\mu(C)$ be the set of non-$\mu$-matches in  $C$
  and let $NM_\mu(C)=|\mathcal{NM}_\mu(C)|$. Let 
\begin{equation}
NM_{n,\mu}(q)=\sum_{C \in \mathcal{C}_n}q^{NM_\mu(C)}.
\end{equation} 
Note that for any $C \in \mathcal{C}_n$, $NM_\mu(C)+NT_\mu(C)={n-1 \choose 2}$ 
so that \\
$NM_{n,\mu}(q)|_{q^k}=NT_{n,\mu}(q)|_{q^{{n-1\choose2}-k}}$. Table \ref{tab:NM} shows the polynomials
$NM_{n,\mu}(q)$ for $1\le n \le 10$. Our data suggests the following theorem.

 \begin{table}[h]
\caption{Polynomials $NM_{n,\mu}(q)$}\label{tab:NM}    
\begin{tabular}{c||*{2}{l}}
$n=1$ & $1$ \\
 \hline
$2$ & $1$ \\
 \hline
$3$ & $1 + q$ \\
 \hline
$4$ & $1 + q + 3\,q^2 + q^3$ \\
 \hline
$5$ & $1 + q + 2\,q^2 + 7\,q^3 + 6\,q^4 + 6\,q^5 + q^6$ \\
 \hline
$6$ & $1 + q + 2\,q^2 + 3\,q^3 + 13\,q^4 + 15\,q^5 + 23\,q^6 + 31\,q^7 + 20\,q^8 + 10\,q^9 + q^{10}$ \\
 \hline
$7$ & $1 + q + 2\,q^2 + 3\,q^3 + 5\,q^4 + 19\,q^5 + 25\,q^6 + 46\,q^7 + 66\,q^8 + 119\,q^9 + 126\,q^{10} +$\\
&$ 135\,q^{11} + 106\,q^{12} + 50\,q^{13} + 15\,q^{14} + q^{15}$ \\
 \hline
$8$ & $1 + q + 2\,q^2 + 3\,q^3 + 5\,q^4 + 7\,q^5 + 27\,q^6 + 33\,q^7 + 65\,q^8 + 101\,q^9 + 174\,q^{10} + $\\
&$299\,q^{11} + 418\,q^{12} + 603\,q^{13} + 726\,q^{14} + 850\,q^{15} + 736\,q^{16} + 561\,q^{17} + $\\
&$301\,q^{18} + 105\,q^{19} + 21\,q^{20} + q^{21}$ \\
 \hline
$9$ & $1 + q + 2\,q^2 + 3\,q^3 + 5\,q^4 + 7\,q^5 + 11\,q^6 + 35\,q^7 + 44\,q^8 + 80\,q^9 + 126\,q^{10} + $\\
&$217\,q^{11} + 338\,q^{12} + 646\,q^{13} + 888\,q^{14} + 1461\,q^{15} + 2116\,q^{16} + 3093\,q^{17} + $\\
&$4055\,q^{18} + 5007\,q^{19} + 5675\,q^{20} + 5541\,q^{21} + 4820\,q^{22} + 3311\,q^{23} + 1870\,q^{24} + $\\
&$742\,q^{25} + 196\,q^{26} + 28\,q^{27} + q^{28}$ \\
 \hline
$10$ & $1 + q + 2\,q^2 + 3\,q^3 + 5\,q^4 + 7\,q^5 + 11\,q^6 + 15\,q^7 + 46\,q^8 + 56\,q^9 + 100\,q^{10} + $\\
&$148\,q^{11} + 251\,q^{12} + 374\,q^{13} + 640\,q^{14} + 1098\,q^{15} + 1640\,q^{16} + 2568\,q^{17} + $\\
&$3971\,q^{18} + 6179\,q^{19} + 9137\,q^{20} + 13710\,q^{21} + 18551\,q^{22} + 25689\,q^{23} + $\\
&$32781\,q^{24} + 40008\,q^{25} + 44119\,q^{26} + 45433\,q^{27} + 41488\,q^{28} + 32864\,q^{29} +$\\
&$ 22392\,q^{30} + 12253\,q^{31} + 5328\,q^{32} + 1638\,q^{33} + 336\,q^{34} + 36\,q^{35} + q^{36}$ \\
 \hline
\end{tabular}
\end{table}

\begin{theorem}{\label{nm}}
For $k<n-2$, $NM_{n,\mu}(q)|_{q^k}=a(k)$  where $a(k)$ is the number of integer partitions of $k$.
\end{theorem}

To prove Theorem \ref{nm}, we need to define what we call 
the plot of non-$\mu$-matches in $C \in \mathcal{C}_n$. We consider 
an $n \times n$ grid where the rows are labeled with $1,2, \ldots , n$, 
reading from bottom to top, and the columns are labeled with $1,2, \ldots, n$, 
reading from left to right. The cell $(i,j)$ is the cell which lies 
in the $i$th row and $j$th column. Then, given a 
set  $S$ of ordered pairs $\langle i,j\rangle $ with $1 \leq i  \leq n$, 
we let 
 $plot_n(S)$ denote the diagram that arises by shading 
a cell $(j,i)$ on the $n \times n$ grid if and only if $\langle i,j\rangle  \in S$. Given an $n$-cycle $C \in \mathcal{C}_n$, we 
let $NMplot(C) = plot_n(\mathcal{NM}_\mu(C))$. 
For example, if $C=(1,4,5,3,8,7,2,6)$, then 
$$\mathcal{NM}_\mu(C)=\{\langle 1,5\rangle ,\langle 1,6\rangle ,\langle 1,7\rangle ,\langle 1,8\rangle ,\langle 2,5\rangle ,\langle 2,7\rangle ,\langle 2,8\rangle ,\langle 3,5\rangle ,\langle 4,6\rangle ,\langle 4,7\rangle ,\langle 4,8\rangle \}$$
and $NMplot(C)$ is pictured in Figure \ref{fig:nmplot}. 

\begin{figure}[htbp]
  \begin{center}
    \includegraphics[width=0.25\textwidth]{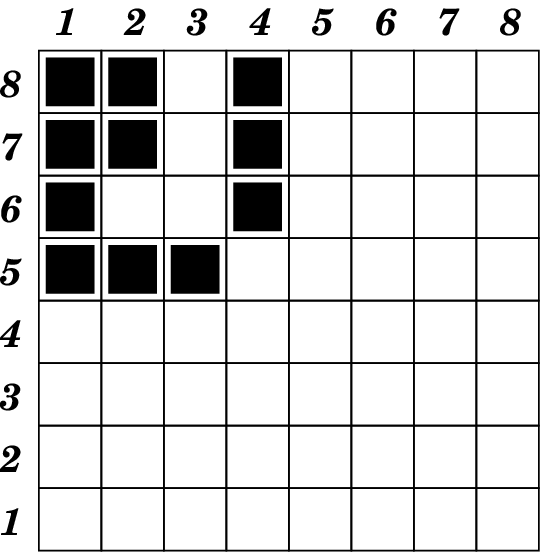}
    \caption{$NMplot((1,4,5,3,8,7,2,6))$.}
    \label{fig:nmplot}
  \end{center}
\end{figure}

First, we observe that if $C \in \mathcal{C}_n$, then 
$\mathcal{NM}_\mu(C)$ completely determines $C$. That is, we 
have the following theorem. 

\begin{theorem}
If $S$ is a set of ordered pairs such that $\mathcal{NM}_\mu(C)=S$ for some $n$-cycle $C \in \mathcal{C}_n$, then
$C$ is the only $n$-cycle such that $\mathcal{NM}_\mu(C)=S$.
\end{theorem}
\begin{proof}
Our proof proceeds by induction on $n$. Clearly, the theorem holds 
for $n =1$ and $n =2$. Now, assume that the theorem holds for 
$n$.  Let $C' \in \mathcal{C}_{n+1}$ and  
$S = \mathcal{NM}_{\mu}(C')$. Let $C= (1,c_2, \ldots, c_n)$ 
be the $n$-cycle that is obtained 
from $C$ by removing $n+1$.  Then it is easy to see that 
$\mathcal{NM}_{\mu}(C) = S- \{\langle i,n+1\rangle :\langle i,n+1\rangle  \in S\}$. Hence, $C$ is 
the unique $n$-cycle in $\mathcal{C}_n$ such that $\mathcal{NM}_{\mu}(C) = S- \{\langle i,n+1\rangle :\langle i,n+1\rangle  \in S\}$.  Let $C^{(i)}$ denote the cycle that 
results by inserting $n+1$ immediately after $i$ in $C$.  Since $C$ is 
unique, it follows that $C'$ must equal $C^{(i)}$ for some $1 \leq i \leq n$.
However, it is easy to see that if $1 \leq i <n$, then 
$\langle j,n+1\rangle  \in \mathcal{NM}_{\mu}(C^{(i)})$ for 
$1 \leq j <i$ and $\langle i,n+1\rangle  \not \in  \mathcal{NM}_{\mu}(C^{(i)})$. 
If $i =n$, then $\langle j,n+1\rangle  \in \mathcal{NM}_{\mu}(C^{(i)})$ for 
$1 \leq j <n$.  Thus, it follows that $\mathcal{NM}_\mu(C^{(1)}), \ldots, \mathcal{NM}_\mu(C^{(n)})$ are pairwise distinct. Hence, there is exactly 
one cycle $C'$ such that $\mathcal{NM}_\mu(C') = S$. 
\end{proof}

An integer sequence  $\lambda=(\lambda_1,\lambda_2,\dots,\lambda_\ell)$ 
is a {\em partition of $n$} if $\lambda_1 \geq \lambda_2 \geq \cdots \geq \lambda_{\ell} > 0$ and $\sum_{i=1}^{\ell} \lambda_i =n$.  
If 
$\lambda=(\lambda_1,\lambda_2,\dots,\lambda_\ell)$ is a partition of 
$n$, we will write $\lambda \vdash n$, and let $|\lambda| =n$ denote the size 
of $\lambda$ and 
$\ell(\lambda) = \ell$ denote the length of $\lambda$.  The 
\emph{Ferrers diagram} of a partition $\lambda$  on the $m \times m$ grid, denoted $FD_m(\lambda)$, is the diagram that results by shading 
the squares $(1,m),(2,m),\dots,(\lambda_1,m)$ in the top row, 
shading the squares $(1,m-1),(2,m-1),\dots,(\lambda_2,m-1)$ in the second 
row from the top, and, in general, shading the 
squares $(1,m-i+1),(2,m-i+1),\dots,(\lambda_i,m-i+1)$ in the 
$i$th row from the top. For example, the Ferrers diagram 
of $\lambda=(5,4,3,1,1)$ on a $8 \times 8$ grid is pictured in Figure \ref{fig:ferrers}.

\begin{figure}[htbp]
  \begin{center}
    \includegraphics[width=0.25\textwidth]{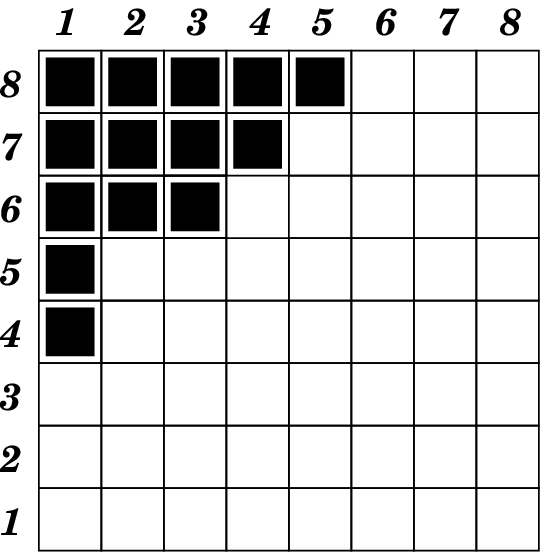}
    \caption{ $FD_8(5,4,3,1,1)$}
    \label{fig:ferrers}
  \end{center}
\end{figure}

It is easy to see that the following three properties characterize 
the Ferrers diagrams of partitions $\lambda$ in an $n\times n$ grid.
\begin{enumerate}
\item
If $\lambda$ is not empty, then the square $(n,1)$ is shaded in 
$FD_m(\lambda)$.
\item
If $x \neq 1$ and the square $(y,x)$ is shaded in $FD_m(\lambda)$, 
then the square $(y,x-1)$ is shaded in $FD_m(\lambda)$.
\item
If $y \neq n$ and the square $(y,x)$ is shaded in $FD_m(\lambda)$, 
 then the square $(y+1,x)$ is shaded in $FD_m(\lambda)$.
\end{enumerate}
This means that if $S$ is a set of pairs $(i,j)$ such 
that $plot_n(S)$ is a Ferrers diagram, then it must be the case 
that 
\begin{enumerate}
\item
If $S$ is not empty, then the square $(1,n) \in S$.
\item
If $x \neq 1$ and the square $(x,y) \in S$,  
then $(x-1,y) \in S$.
\item
If $y \neq n$ and the square $(x,y) \in S$, 
 then $(x,y+1) \in S$.
\end{enumerate}

Our next lemma and corollary will show that if 
$C \in \mathcal{C}_n$ and $NM_\mu(C) < n-2$, then 
$\mathcal{NM}_\mu (C)$ has the same three properties. 

\begin{lemma}\label{lem:n-2}
Assume that $C \in \mathcal{C}_n$. 
\begin{enumerate}
\item
If $\langle 1,n\rangle \notin\mathcal{NM}_\mu(C)$ and ${NM}_\mu(C)\neq 0$, then $NM_\mu(C)\ge n-2$.
\item
If $x\ne 1$, $\langle x,y\rangle \in\mathcal{NM}_\mu(C)$, and $\langle x-1,y\rangle \notin\mathcal{NM}_\mu(C)$, then $NM_\mu(C)\ge n-2$.
\item
If $y\ne n$, $\langle x,y\rangle \in\mathcal{NM}_\mu(C)$, and $\langle x,y+1\rangle \notin\mathcal{NM}_\mu(C)$, then $NM_\mu(C)\ge n-2$.
\end{enumerate}
\end{lemma}
\begin{proof}
Let $C =(1,c_1, \ldots, c_{n-1}) \in \mathcal{C}_n$  be an $n$-cycle 
where $n\ge4$.

For part (1), suppose that $\langle 1,n\rangle \notin\mathcal{NM}_\mu(C)$ and ${NM}_\mu(C)\neq 0$. Then, there are no integers that are cyclically in between $1$ and $n$ in 
$C$ so that $c_1 =n$. 
Since ${NM}_\mu(C)\neq0$, $C\ne(1,n,n-1,\dots,3,2)$. Thus, $C$ must 
be of the form 
$$C=(1,n,n-1,\dots,n-a,b,\underbrace{\dots}_A b+1 \underbrace{\dots}_B)$$
where $b+1 < n-a$. That is, the sequence $c_1>c_2 > \cdots > c_{a+1}$ 
consists of the decreasing interval from $n$ to $n-a$ and then $b\leq n-a-2$. 
It follows that the pairs $\langle 1,b+1\rangle ,\langle b,n\rangle ,\dots,\langle
b,n-a\rangle $ are all 
\emph{non-$\mu$-matches} in $C$ which accounts for 
$a+2$ non-$\mu$-matches in $C$. Let $A$ be the set of integers cyclically between $b$ and $b+1$ in $C$ and
let $B$ be the set of integers cyclically between $b+1$ and $1$ in $C$. 
Note that 
\begin{enumerate}
\item
if $x\in A$ and $x<b$, then $\langle x,n\rangle \in\mathcal{NM}_\mu(C)$, 
\item
if $x\in A$ and $x>b$, then $\langle 1,x\rangle \in\mathcal{NM}_\mu(C)$,
\item
if $x\in B$ and $x<b$, then $\langle x,b+1\rangle \in\mathcal{NM}_\mu(C)$, and 
\item
if $x\in B$ and $x>b$, then $\langle 1,x\rangle \in\mathcal{NM}_\mu(C)$.
\end{enumerate}
It follows that each element in $A$ and $B$ is part of  at least 
one non-$\mu$-match in $C$ so that we know that 
$NM_\mu(C) \geq a+2 + |A| +|B|$. Thus, since $|A|+|B|=n-a-4$,  we have
$NM_\mu(C) \geq n-2$. 
\\ 

For part (2), suppose that $x\ne 1$, $\langle x,y\rangle \in\mathcal{NM}_\mu(C)$, and 
 $\langle x-1,y\rangle \notin\mathcal{NM}_\mu(C)$.
Then $x$ cannot be cyclically between $x-1$ and $y$ in $C$. 
Also, there exists an integer $z$ with
$x<z<y$ such that
$z$ is cyclically between $x$ and $y$ in $C$, but $z$ is not cyclically between $x-1$ and $y$ in $C$. It follows that $C$ is of the form 
$$C=(\underbrace{1,\dots }_{A_3},x, \underbrace{\dots, z, \dots}_{A_1}, x-1, \underbrace{\dots}_{A_2}, y, \underbrace{\dots}_{A_3})$$
where $A_1$ is the set of integers cyclically between $x$ and $x-1$ in $C$ 
that are not equal to $z$, 
$A_2$ is 
the set of integers cyclically between $x-1$ and $y$ in $C$, and 
$A_3$ is the set of integers between $y$ and $x$ in $C$. Note that 
$\langle x,y\rangle $ and $\langle x-1,z\rangle $ are in $\mathcal{NM}_\mu(C)$. Moreover,  
\begin{enumerate}
\item
if $a\in A_1$ and $a<x-1$, then $\langle a,y\rangle \in\mathcal{NM}_\mu(C)$,
\item
if $a\in A_1$ and $a>x$, then $\langle x-1,a\rangle \in\mathcal{NM}_\mu(C)$,
\item
if $a\in A_2$ and $a<x-1$, then $\langle a,z\rangle \in\mathcal{NM}_\mu(C)$,
\item there are no $a\in A_2$ with $x-1<a<y$ 
since $\langle x-1,y\rangle \notin\mathcal{NM}_\mu(C)$, 
\item
if $a\in A_2$ and $a>y$, then $\langle 1,a\rangle \in\mathcal{NM}_\mu(C)$, 
\item
if $a\in A_3$ and $a<z$, then $\langle a,y\rangle \in\mathcal{NM}_\mu(C)$, and 
\item
if $a\in A_3$ and $a>z$, then $\langle x,a\rangle \in\mathcal{NM}_\mu(C)$.
\end{enumerate}
Therefore, any integer $a\in A_1\cup A_2\cup A_3$ is part of a distinct 
non-$\mu$-match in $C$. Since \\
$|A_1\cup A_2\cup A_3|=n-4$, it 
follows that $NM_\mu(C) \geq 2 + n-4 = n-2$.

For part (3), suppose that $y\ne n$, $\langle x,y\rangle \in\mathcal{NM}_\mu(C)$, 
and $\langle x,y+1\rangle \notin\mathcal{NM}_\mu(C)$.
Then $y$ cannot be cyclically between $x$ and $y+1$ in $C$. Also, there exists an integer $z$ with 
$x<z<y$ such that
$z$ is cyclically between $x$ and $y$ in $C$, but $z$ is not cyclically between $x$ and $y+1$ in $C$. Thus, $C$ must be of the form 
$$C=(\underbrace{1,\dots }_{B_3},x, \underbrace{\dots}_{B_1}, y+1, \underbrace{\dots, z, \dots}_{B_2}, y, \underbrace{\dots}_{B_3})$$
where $B_1$ is the set of integers cyclically between $x$ and $y+1$ in $C$, 
$B_2$ is the set of integers cyclically between $y+1$ and $y$ in $C$ that 
are not equal to $z$, and 
$B_3$ is the set of integers cyclically between $y$ and $x$ in $C$. Note that 
$\langle x,y\rangle $ and $\langle z,y+1\rangle$ are in $\mathcal{NM}_\mu(C)$.  Moreover, 
\begin{enumerate}
\item
if $b\in B_1$ and $b<x$, then $\langle b,y\rangle \in\mathcal{NM}_\mu(C)$, 
\item there is no $b\in B_1$ with $x<b<y+1$ 
since $\langle x,y+1\rangle \notin\mathcal{NM}_\mu(C)$, 
\item
if $b\in B_1$ and $b>y+1$, then $\langle z,b\rangle \in\mathcal{NM}_\mu(C)$,
\item
if $b\in B_2$ and $b<y$,  then $\langle b,y+1\rangle \in\mathcal{NM}_\mu(C)$,
\item
if $b\in B_2$ and $b>y+1$, then $\langle x,b\rangle \in\mathcal{NM}_\mu(C)$,
\item
if $b\in B_3$ and $b<z$, then $\langle b,y\rangle \in\mathcal{NM}_\mu(C)$, and 
\item
if $b\in B_3$ and $b>z$ then $\langle x,b\rangle \in\mathcal{NM}_\mu(C)$.
\end{enumerate}
Therefore, any integer $b\in B_1\cup B_2\cup B_3$ is part of a distinct 
non-$\mu$-match in $C$. Since 
$|B_1\cup B_2\cup B_3|=n-4$, it follows that 
$NM_\mu(C)\geq 2+ n-4 =n-2$.
\end{proof}

\begin{corollary}\label{nmplot}
If $NM_\mu(C)< n-2$, then $NMplot(C)$ is a Ferrers diagram.
\end{corollary}
\begin{proof}
This corollary follows directly from Lemma \ref{lem:n-2}. That is, 
suppose that $C \in \mathcal{C}_n$ and $NM_\mu(C) < n-2$. First, 
if $NM_\mu(C) =0$, then $C = (1,n,n-1, \ldots, 2)$ in which 
case $FD_n(C)$ has no shaded squares which correspond to the empty partition. 
Thus, assume that $1 \leq NM_\mu(C) < n-2$. Then, it 
follows from part (1) of Lemma \ref{lem:n-2} that  
$\langle 1,n\rangle \in\mathcal{NM}_\mu(C)$. Next it follows from part (2) of Lemma \ref{lem:n-2} that if $x\ne 1$ and $\langle x,y\rangle \in\mathcal{NM}_\mu(C)$, then 
$\langle x-1,y\rangle \in\mathcal{NM}_\mu(C)$. Finally, it follows from part (3) 
of  Lemma \ref{lem:n-2} that if  $y\ne n$ and $\langle x,y\rangle \in\mathcal{NM}_\mu(C)$, 
then $\langle x,y+1\rangle \in\mathcal{NM}_\mu(C)$. Hence, the shaded 
cells $NMplot(C)$ must be a 
Ferrers diagram of a partition $\lambda$ of $NM_\mu(C)$. 
\end{proof}

If $n \geq 3$, we let 
 $T_n$ be the plot of the Ferrers diagram in the $n \times n$ 
grid corresponding the to partition $\lambda = (n-2,n-3, \ldots, 1)$. Thus, for 
example, $T_7$ is pictured in Figure \ref{fig:T7}. One can see that the sets that $T_n$ have the property that
$|\{\lambda: \lambda \text{ is a partition and }FD_n(\lambda)\subseteq T_n\}|=C_n$ where $C_n$ is the $n$th Catalan number since the lower boundaries 
of the plots of $FD(\lambda)$ for such $\lambda$ 
correspond to Dyck paths.
 \\

\begin{figure}[htbp]
  \begin{center}
    \includegraphics[width=0.25\textwidth]{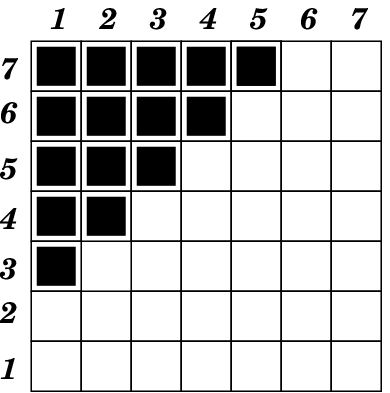}
    \caption{ The Ferrers diagram $T_7$ on the $7 \times 7$ grid.}
    \label{fig:T7}
  \end{center}
\end{figure}

Our next theorem will show that for any $n \geq 3$ and any 
partition  $\lambda\subseteq T_n$, we can construct an 
$n$-cycle $C \in \mathcal{C}_n$ such that 
$NMplot(C) = FD_n(\lambda)$. 

\begin{theorem}\label{thm:const} Suppose that $n \geq 3$ and 
$\lambda$ is a partition such that $\lambda \subseteq T_n$. Then, 
there is an $n$-cycle $C \in \mathcal{C}_n$ such that 
$NMplot(C) = FD_n(\lambda)$.
\end{theorem}
\begin{proof}
We proceed by induction on $n$. For $n=3$, it is 
easy to see that if $C^{(1)}=(1,3,2)$, the $NMplot(C^{(1)})$ is empty and 
if $C^{(2)} =(1,2,3)$, then $NMPlot(C^{(2)}) = T_3$. Note that 
$C^{(1)}$ and $C^{(2)}$ have the property that if the largest part 
of the corresponding partition is of size $i$, then 
$3$ immediately follows $i+1$ in the cycle. 
Thus, our theorem 
holds for $n=3$. 

Now, suppose that $n > 3$ and 
$\lambda = (\lambda_1, \ldots, \lambda_k)$ is 
a partition which is contained in $T_n$. It follows that 
$\lambda^- = (\lambda_2, \ldots, \lambda_k)$ is contained 
in $T_{n-1}$. Assume by induction that there is an 
 $(n-1)$-cycle $C'$ such that $NMplot(C') =FD_{n-1}(\lambda^-)$ 
and $n-1$ immediately follows $\lambda_2+1$ in $C'$.
Thus, in $C'$, the pairs $\langle \lambda_2+1,n-1\rangle , 
\langle \lambda_2+2,n-1\rangle , \ldots, 
\langle n-3,n-1\rangle$ must match $\mu$ in $C'$. This means that 
if  
$\lambda_2 +1 \leq n-3$, then $n-3$ must lie between $n-2$ and $n-1$ in 
$C'$. Next  if  $\lambda_2 +1 \leq n-4$, then $n-4$ must lie 
between $n-3$ and $n-1$ in $C'$. In general, if 
$\lambda_2 +1 \leq n-k$, then $n-k$ must lie 
between $n-k+1$ and $n-1$ in $C'$. Thus, $C'$ must be of the following 
form  
$$C'=(\underbrace{\dots}_{A_{n-1}} n-2 \underbrace{\dots}_{A_{n-2}} n-3  
\underbrace{\dots}_{A_{n-3}} n-4 \underbrace{\dots}_{A_{n-4}} 
\cdots  (\lambda_2+2)\underbrace{\dots}_{A_{\lambda_2  +2}} (\lambda_2+1) ,
(n-1)  
\underbrace{\dots}_{A_{n-1}}).$$ 
where $A_i$ is the set of elements cyclically between $n-i$ and $n-i-1$ in 
$C'$ for 
$i =1, \ldots, n-\lambda_2-2$. 
Now let $C$ be the cycle that results from $C'$ by inserting 
$n$ immediately after $\lambda_1+1$ in $C'$. Inserting 
$n$ into $C'$ does not effect on whether pairs 
$\langle i,j\rangle $ with $1\leq i < j \leq n-1$ are 
$\mu$-matches in $C$. That 
is, for such pairs $\langle i,j\rangle  \in \mathcal{NM}_\mu(C)$ 
if and only if 
$\langle i,j\rangle  \in \mathcal{NM}_\mu(C')$. Thus, the diagram of 
$NMplot(C')$ and $NMplot(C)$ are the same up to row $n-1$. Now, in 
row $n$, we know that the cells $(n,1), \ldots, (n,\lambda_1)$ are 
shaded since the fact that $n$ immediately follows $\lambda_1+1$ 
in $C$ means that $\langle i,n\rangle  \in \mathcal{NM}_\mu(C)$ for 
$i=1, \ldots, \lambda_1$. However, it is easy to see 
from the form of $C'$ above that $\langle n-2,n\rangle , \langle n-3,n\rangle , \ldots, 
\langle\lambda_1+1,n\rangle $ are $\mu$-matches in $C$ since 
$n-2, n-3, \ldots, \lambda_1+1$ appear in decreasing order as 
we traverse clockwise around the cycle $C$.  Thus, 
$NMplot(C) = FD_n(\lambda)$. 
\end{proof}

Note the proof of Theorem \ref{thm:const} gives a simple 
algorithm to construct an $n$-cycle 
$C_\lambda$ such that $NMplot(C_\lambda) =FD_n(\lambda)$ for 
any $\lambda \subseteq T_n$. For example, 
suppose that $n =12$ and $\lambda =(3,3,2,1)$ so that the Ferrers 
diagram in the $12 \times 12$ grid is pictured in Figure \ref{fig:ex}.

\begin{figure}[htbp]
  \begin{center}
    \includegraphics[width=0.3\textwidth]{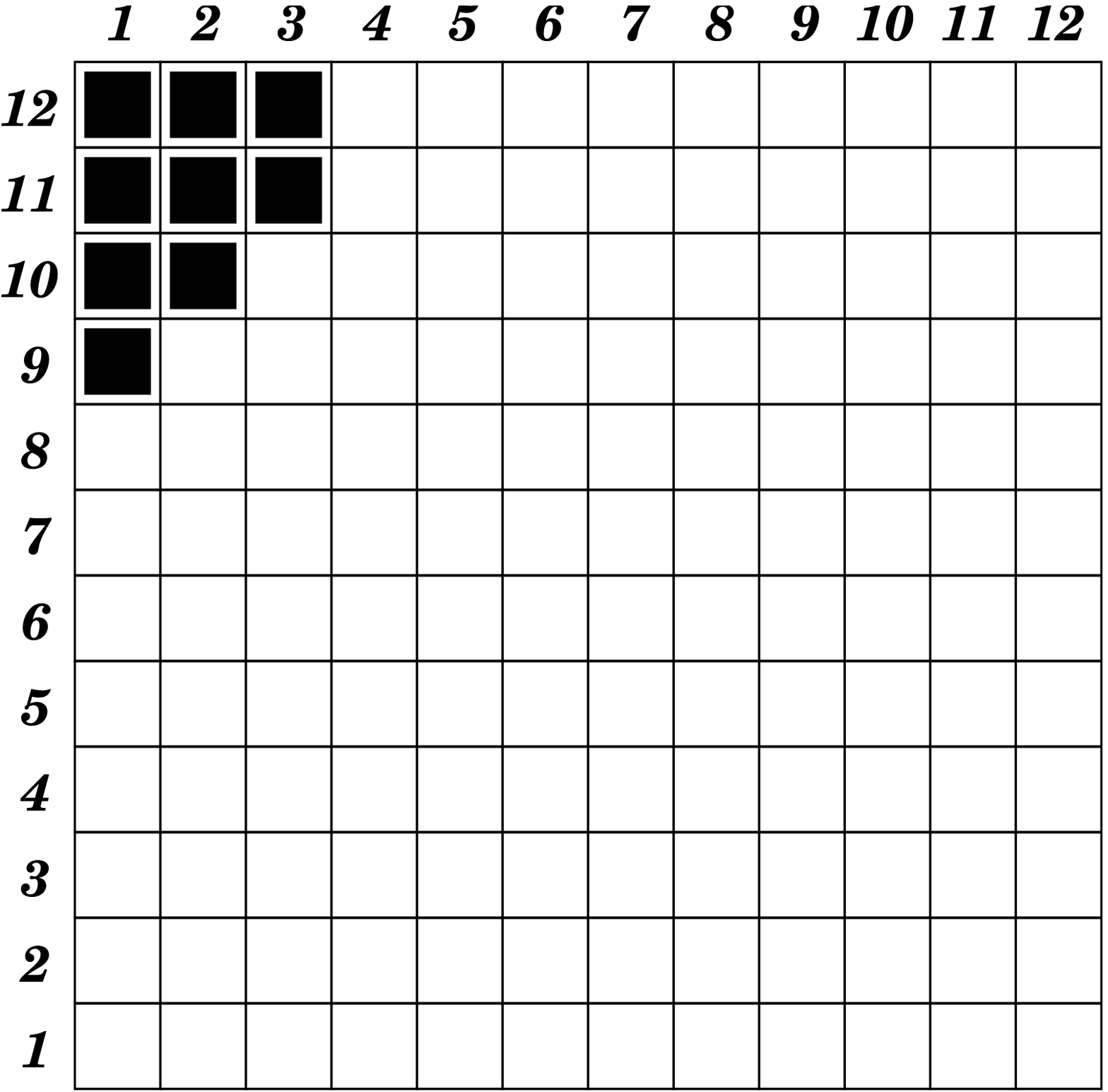}
    \caption{ $sh_{12}(3,3,2,1)$.}
    \label{fig:ex}
  \end{center}
\end{figure}

Since there are no non-$\mu$-matches in the first eight rows, we 
must start with the cycle $C_8=(1,8,7,6,5,4,3,2)$. Then, our 
proof of Theorem \ref{thm:const} tells us that we 
should build up the cycle structure by first 
creating a cycle $C_9 \in \mathcal{C}_9$ by inserting 
$9$ immediately after 2 in $C_8$, since the number of 
non-$\mu$-matches in row 9 is 1.  Then we 
create a cycle $C_{10} \in \mathcal{C}_{10}$ by inserting 
$10$ immediately after 3 in $C_9$ since the number of 
non-$\mu$-matches in row 10 is 2. Then, we 
create a cycle $C_{11} \in \mathcal{C}_{11}$ by inserting 
$11$ immediately after 4 in $C_{10}$ since the number of 
non-$\mu$-matches in row 11 is 3. Finally, we 
create a cycle $C_{12} = C_\lambda \in \mathcal{C}_{12}$ by inserting 
$12$ immediately after 4 in $C_{11}$ since the number of 
non-$\mu$-matches in row 12 is 3. Thus, 

\begin{eqnarray*}
C_9 &=& (1,8,7,6,5,4,3,2,9), \\
C_{10} &=& (1,8,7,6,5,4,3,10,2,9),\\
C_{11} &=& (1,8,7,6,5,4,11,3,10,2,9),\ \mbox{and} \\
C_{12} &=& (1,8,7,6,5,4,12,11,3,10,2,9).
\end{eqnarray*}
Then we have that  
\begin{align*}
\mathcal{NM}_\mu((1,8,7,6,5,4,12,11,3,10,2,9))=&\langle1,12\rangle,\langle2,12\rangle,\langle3,12\rangle,\\
&\langle1,11\rangle,\langle2,11\rangle,\langle3,11\rangle,\\
&\langle1,10\rangle,\langle2,10\rangle,\\
&\langle1,9\rangle.
\end{align*}

Now we can prove Theorem \ref{nm}.
\begin{proof} Suppose that $k \leq n-2$. Let
\begin{eqnarray*}  
FD_n(k)&=& \{C \in \mathcal{C}_n: NMplot(C) = FD_n(\lambda) \ \mbox{for some} \  \lambda\vdash k\} \ \mbox{and} \\
NMp_n(k)&=&\{C \in \mathcal{C}_n: NM_\mu(C) =k\}.
\end{eqnarray*}
 Theorem \ref{nmplot} shows that $NMp_n(k) \subseteq FD_n(k)$ 
and Theorem \ref{thm:const} shows that 
$FD_n(k) \subseteq NMp_n(k)$. Thus, 
$NM_{n,\mu}(q)|_{q^k} =  |NMp_n(k)| = |FD_n(k)|$. Hence, 
$NM_{n,\mu}(q)|_{q^k}$ equals the number of partitions of $k$.
\end{proof}
 Now, we will show that for $k<n-2$, $NTI_{n,\mu}(q)|_{q^{\binom{n-1}{2} -k}}=NT_{n,\mu}(q)|_{q^{\binom{n-1}{2} -k}}=a(k)$, where
$a(k)$ is the number of 
partitions of $k$. First we shall show that  if $C \in \mathcal{C}_n$ 
has fewer than $n-2$ non-matches, then $C$ must be incontractible. 

\begin{lemma}\label{n-2in}
If $C \in \mathcal{C}_n$ 
has fewer than $n-2$ non-matches, then $C$ must be incontractible. 
\end{lemma}

\begin{proof}
If $i+1$
immediately follows $i$ in an $n$-cycle $C \in \mathcal{C}_n$, then,
clearly, $\langle j,i+1 \rangle$ are non-$\mu$-matches for $1 \leq j < i$
and $\langle i, k \rangle$ is a non-$\mu$-match for $i+2 \leq k \leq n$.
This gives $n-2$ non-$\mu$-matches.
\end{proof}

\begin{corollary}
For $k<n-2$, $NTI_{n,\mu}(q)|_{q^{\binom{n-1}{2} -k}}=NT_{n,\mu}(q)|_{q^{\binom{n-1}{2} -k}}=a(k)$  where
$a(k)$ is the number of 
partitions of $k$.
\end{corollary}

\begin{proof}
By definition, $NT_{n,\mu}(q)|_{q^{{n-1\choose2}-k}}=NM_{n,\mu}(q)|_{q^k}$. Thus, by Theorem
\ref{nm}, for $k<n-2$, $NT_{n,\mu}(q)|_{q^{\binom{n-1}{2} -k}}=a(k)$.

By Lemma \ref{n-2in}, we have that if a cycle $C$ has $k$ non-$\mu$-matches 
where $k<n-2$,  then $C$ is 
incontractible. It follows that if a cycle has ${\binom{n-1}{2} -k}$ non-trivial $\mu$-matches, then it is 
incontractible. Thus, for $k<n-2$, $NTI_{n,\mu}(q)|_{q^{\binom{n-1}{2} -k}}=NT_{n,\mu}(q)|_{q^{\binom{n-1}{2} -k}}$. And so $NTI_{n,\mu}(q)|_{q^{\binom{n-1}{2} -k}}=a(k)$ for $k<n-2$.
\end{proof}

 \section{Conclusions and direction for further research}\label{concl}
 
 In this paper, we studied the polynomials 
$NT_{n,\mu}(q) = \sum_{C \in \mathcal{C}_n} q^{NT_\mu(C)}$ and 
$NTI_{n,\mu}(q) = \sum_{C \in \mathcal{IC}_n} q^{NT_\mu(C)}$. 
We showed that $NTI_{n,\mu}(1)$ is the number of derangements of  
$S_{n-1}$. Thus, the polynomial $NTI_{n,\mu}(q)$ is a $q$-analogue 
of the derangement number $D_{n-1}$.  There are several 
$q$-analogues of the derangement numbers that have been studied 
in the literature, see the papers  by Garsia 
and Remmel \cite{GarRem} and Wachs\cite{Wachs}.  Our 
$q$-analogue of $D_{n-1}$ is different from either 
the Garsia-Remmel $q$-analogue or the Wachs $q$-analogue of 
the derangement numbers. Moreover, we proved that 
$$NT_{n,\mu}(q)|_{q^k}=\sum_{s=1}^{2k+1} \sum_{\substack{A \in \mathcal{IC}_s,\\NT_\mu(A)=k}}
{n-1\choose s-1}$$
so that the coefficients of the polynomial $NT_{n,\mu}(q)$ can 
be expressed in terms of the coefficients of the polynomials 
$NTI_{j,\mu}(q)$.  We also showed that $NTI_{n,\mu}(q)|_{q^{\binom{n-1}{2}-k}}$
equals the number of partition of $k$ for $k < n-2$.

The main open question is to find some sort of recursion 
or generating function that would allow us to compute $NTI_{n,\mu}(q)$. 
Note that several of the sequences 
$(NT_{n,\mu}(q)|_{q^k})_{n \geq 2}$ appear in the OEIS \cite{oeis}. 
For example, we showed that $ NT_{n,\mu}(q)|_{q^2} = \binom{n-1}{3} + 
2\binom{n-1}{4}$ so that the sequence $(NT_{n,\mu}(q)|_{q^2})_{n \geq 4}$ 
starts out $1,6,20,50,105,196,336,540, 825, \ldots$.  The 
$n$th term of this sequence has several combinatorial interpretations 
including being the number of $\sg \in S_n$ which are $132$-avoiding and 
have exactly two descents, the number of Dyck paths on length $2n+2$ with 
$n-1$ peaks,  and the number of squares with corners on the 
$n \times n$ grid.  Thus, it would be interesting to find 
bijections from these objects to our $(n+4)$-cycles $C \in \mathcal{C}_n$ 
such that 
$NT_\mu(C) =2$. We also  showed that $ NT_{n,\mu}(q)|_{q^3} = \binom{n-1}{3} + 
3\binom{n-1}{4}+6\binom{n-1}{5}+5\binom{n-1}{6}$ so that the sequence $(NT_{n,\mu}(q)|_{q^3})_{n \geq 4}$ 
starts out $$1,7,31,102,273,630,1302,2472, 4389 \ldots.$$ This sequences 
does not appear in the OEIS. Similarly, the sequence $(NT_{n,\mu}(q)|_{q^4})_{n \geq 5}$ starts out $2,23,135,561,1870,5328,13476, \ldots $ and it does not 
appear in the OEIS.

Finally, it would be interesting to characterize 
the charge graphs of those cycles $C \in \mathcal{IC}_{2n}$ for which  $NM_\mu(C) =n$.  One can see from our tables that sequence 
$(NTI_{2n,\mu}(q)|_{q^n})_{n \geq 2}$ starts out 
$1,6,29,130, \ldots $. Moreover, we have computed 
the $NTI_{12,\mu}(q)|_{q^6} =562$.  This suggests that 
this sequence is sequence A008549 in the OEIS.  If so, 
this would mean that  
$NTI_{2n,\mu}(q)|_{q^n} =\sum_{i=0}^{n-2} \binom{2n-1}{i}$ for 
$n \geq 2$.


\begin{thebibliography}{10}


\bibitem{AKV} S. Avgustinovich, S. Kitaev and A. Valyuzhenich, Avoidance of boxed mesh patterns on permutations, {\em Discrete Appl. Math.}, {\bf 161} (2013) 43--51.

\bibitem{BrCl} P. Br\"and\'en and A. Claesson, Mesh patterns and the expansion of permutation statistics as sums of permutation patterns, {\em Elect. J. Comb.}, {\bf 18(2)} (2011), \#P5, 14pp. 

\bibitem{GarRem} A.M. Garsia and J.B. Remmel, A combinatorial 
interpretation of $q$-derangement and $q$-Laguerre numbers, 
European J. Comb., {\bf 1} (1980), 47--59. 




\bibitem{JR} M. Jones and J. Remmel, Pattern matching in the cycle structure of permutations, {\em Pure Math. and Appl.}, 
{\bf 22} (2011), 173--208.

\bibitem{kit} S. Kitaev, Patterns in permutations and words, Springer-Verlag, 2011.

\bibitem{kitlie} S. Kitaev and J. Liese, Harmonic numbers, Catalan triangle and mesh patterns, {\em Discrete Math.}, {\bf 313} (2013) 1515--1531.


\bibitem{LS} A. Lascoux and M.P. Sch\"utzenbeger, Sur une conjecture de 
H.0. Foulkes, C.R. Acad. Sci. Paris S\'er I Math., {\bf 288} (1979), 95--98.

\bibitem{oeis} N.~J.~A.~Sloane, The on-line encyclopedia of integer sequences,
published electronically at \phantom{*} {\tt
http://oeis.org}.

\bibitem{Wachs} M. Wachs, On $q$-derangement numbers, {\em Proc. Amer. Math. Soc.}, 
{\bf 106}(1) (1989), 273--278. 

\end{thebibliography}
\end{document}